\documentclass{amsart}
\usepackage{fullpage}
\usepackage{yhmath}  
\usepackage{amsmath,amsthm}
\usepackage{marvosym}
\usepackage{graphicx}
\usepackage{url}
\usepackage{color}
\usepackage{xfrac} 

\usepackage{booktabs}

\newcommand{\otoprule}{\midrule[\heavyrulewidth]}

\graphicspath{{./Figures/}}

\theoremstyle{plain}
\newtheorem{theorem}{Theorem}[section]

\newtheorem{cor}[theorem]{Corollary}
\newtheorem{conj}[theorem]{Conjecture}

\newtheorem{lemma}[theorem]{Lemma}

\newtheorem{addendum}[theorem]{Addendum}

\newtheorem{question}[theorem]{Question}
\newtheorem{dfn}[theorem]{Definition}
\newtheorem*{theoremfillinglengths}{Theorem~\ref{thm:fillinglengths}}
\newtheorem*{theoremjpequivalence}{Theorem~\ref{thm:jpequivalence}}

\newtheorem*{claim*}{Claim}

\theoremstyle{definition}

\newtheorem{remark}[theorem]{Remark}

\newcommand{\hatQ}{\ensuremath{{\widehat{\Q}}}}
\newcommand{\hatZ}{\ensuremath{{\widehat{\Z}}}}

\newcommand{\calB}{\ensuremath{{\mathcal B}}}
\newcommand{\calL}{\ensuremath{{\mathcal L}}}

\newcommand{\bdry}{\ensuremath{\partial}}

\DeclareMathOperator{\Isom}{Isom}

\newcommand{\nbhd}{\ensuremath{\mathcal{N}}}

\newcommand{\Q}{\ensuremath{\mathbb{Q}}}
\newcommand{\R}{\ensuremath{\mathbb{R}}}
\newcommand{\Z}{\ensuremath{\mathbb{Z}}}

\newcommand{\mobius}{M\"{o}bius }

\title{Jointly primitive knots and surgeries between lens spaces.}
\author{Kenneth L.\ Baker, Neil R.\ Hoffman, Joan E.\ Licata  }

\address{Department of Mathematics, University of Miami, 
Coral Gables, FL 33146, USA}
\email{k.baker@math.miami.edu}

\address{Department of Mathematics, Oklahoma State University,
Stillwater, OK 74074}
\email{neil.r.hoffman@okstate.edu}

\address{Mathematical Sciences Institute, The Australian National University,  Australia}
\email{joan.licata@anu.edu.au}

\begin{document}

\begin{abstract}

This paper describes a Dehn surgery approach to generating asymmetric hyperbolic manifolds with two distinct lens space fillings.  Such manifolds were first identified 
 in \cite{DHL} as the result of a computer search of the SnapPy census, but the current work establishes a topological framework for constructing vastly many more such examples.  We introduce the notion of a \textit{jointly primitive} presentation of a knot and show that a refined version of this condition  ---\textit{longitudinally jointly primitive}--- is equivalent to being surgery dual to a $(1,2)$--knot in a lens space. 
 This generalizes  Berge's equivalence between having a doubly primitive presentation and being surgery dual to a $(1,1)$--knot in a lens space. 
Through surgery descriptions on a  seven-component link in $S^3$,
we provide several explicit multi-parameter infinite families of knots in lens spaces with longitudinal jointly primitive presentations and observe among them all the examples previously seen in \cite{DHL}.  
\end{abstract}

\maketitle
\tableofcontents

\section{Introduction}
In \cite{DHL}, Dunfield-Hoffman-Licata identified 22 manifolds with Heegaard genus $3$ and a pair of lens space fillings.  These examples --- the {\em DHL manifolds} --- were identified via the SnapPy census of hyperbolic manifolds admitting geometric triangulations of 10 or fewer tetrahedra.   Specifically, they searched for $1$--cusped, asymmetric manifolds with the property that two of the shortest slopes in the Euclidean torus cross-section of the cusp provided fillings whose fundamental groups had a single generator \cite{DHL}. Remarkably, these DHL manifolds give the first examples of two phenomena in the subject of Dehn surgery:
\begin{enumerate} 
\item They exhibit the ``genus drop'' phenomenon: a decrease in Heegaard genus  by at least 2 for more than one Dehn filling.
\item They give counterexamples to a ``generalized Berge conjecture'': a pair of surgery dual hyperbolic knots in lens spaces of which neither is a $(1,1)$--knot.
\end{enumerate}
This article arose from the attempt to better understand these phenomena by providing a unifying framework for these sporadic examples found by a computer search. 
What emerged from this effort is a characterization of knots that are framed surgery dual to $(1,2)$--knots in lens spaces,  some constructions of families of such knots in lens spaces that are generically hyperbolic and asymmetric, and an identification of the DHL manifolds among the complements of such knots. 

More precisely, we define a {\em longitudinally jointly primitive presentation} of a framed knot $K$ in Section \ref{sec:jpp} and show the following:
\begin{theoremjpequivalence}
A framed knot $K$ has a longitudinal jointly primitive presentation if and only if its framed surgery dual $K^*$ is a $(1,2)$--knot in a lens space. 
\end{theoremjpequivalence}
\noindent
Thereafter we develop several families of framed knots in lens spaces with longitudinally jointly primitive presentations (detailed in Theorem~\ref{thm:bulk}, Theorem~\ref{thm:bulk2}, and Addendum~\ref{add:t12685}) among whose complements we locate the DHL manifolds as detailed in Table~\ref{table:DHL22}.
\begin{cor}\label{cor:dhl}
The 22 DHL manifolds can be realized as complements of $(1,2)$--knots in lens spaces. 
\end{cor} 

For context, recall the surgery characterization of Berge's doubly primitive property for framed knots and the Berge Conjecture.
\begin{theorem}[{\cite{berge2018some}, see also e.g.\ \cite[Appendix]{saito} and \cite[Lemma 3]{BakerDoleshalHoffman}}]
A framed knot $K$ has a doubly primitive presentation if and only if its framed surgery dual $K^*$ is a $(1,1)$--knot in a lens space. \qed
\end{theorem}
\begin{conj}[The Berge Conjecture \cite{berge2018some, gordon1998dehn}]
	Any framed knot in $S^3$ with a surgery to a lens space is doubly primitive.
\end{conj}
\noindent 
The Berge Conjecture and its generalization to knots in $S^1 \times S^2$ (see \cite{BBLS1xS2} and \cite[Conjecture 1.9]{Greene2013Realization}) may still hold, but the DHL manifolds show that the generalization to knots in lens spaces with lens space surgeries fails.

\bigskip

To describe our families of knots with longitudinally jointly primitive presentations, 
in Section~\ref{sec:families} we introduce the knots $K_k(m,r,s,b)$ in the manifolds $Y(m,r,s,b)$; these are obtained as the image of the knot $K$ after surgery on the link $\calL$ shown in Figure~\ref{fig:surgerylink}.  The components are named on the left  of Figure~\ref{fig:surgerylink}, and the surgery slopes are given in terms of the parameters $m,r,s,b,k$  on the right.  Figure~\ref{fig:masterlinksurface} shows the link $K \cup \calL$ again, but with a twice-punctured torus inducing a $+1$ framing. Performing this framed $(+1)$--surgery on $K$ in addition to the surgery on $\calL$ produces the manifold $Y^*_k(m,r,s,b)$.  For certain constraints on $m,r,s,b$, both $Y(m,r,s,b)$ and $Y^*_k(m,r,s,b)$ are lens spaces for all $k \in \Z$.
Table~\ref{table:DHL22} presents the 22 DHL
manifolds as the exteriors of the knots $K_k(m,r,s,b)$ along with their lens space fillings  $Y(m,r,s,b)$ and $Y^*_k(m,r,s,b)$, proving Corollary~\ref{cor:dhl}.

For our constructed families of knots to be ``interesting'' ---beyond simply accounting for the DHL manifolds--- we need to know that, generically, neither they nor their surgery duals are $(1,1)$--knots.  
As $(1,1)$--knots are strongly invertible 
and the symmetry group of a hyperbolic manifold is its isometry group, it suffices to show that the knots in our families are ``generically'' hyperbolic and without a strong inversion in their isometry group.   Since our knots arise through parameterized surgery descriptions, we must take care in developing the proper sense of ``generic''.  To that end we develop the following definitions and theorem (building upon \cite{hodgsonKerckhoff2008} and \cite{kojima}).

Let the multislope $\psi = (\psi_i)_{i=1}^m$ denote slopes on the cusps of $M$ and let $\Psi_N$ be the set of multislopes whose constituent slopes all have lengths greater than a constant $N$; see \S\ref{sec:isometrygroupsfillings} for details.   

\begin{theoremfillinglengths}
	Let $M$ be an $m$--cusped orientable hyperbolic $3$--manifold. Set $C= 7.5832$.
	\begin{enumerate}
		\item For each $\psi \in \Psi_{C}$, $M(\psi)$ is a hyperbolic $3$--manifold in which the cores of the filling are mutually disjoint simple closed geodesics.
		\item Furthermore, there is a constant $C_M>C$ and depending on $M$ such that for each $\psi \in \Psi_{C_M}$, the cores of the filling constitute the shortest geodesics in $M(\psi)$; in this case Kojima's restriction map $\text{rest} \colon \Isom(M(\psi)) \hookrightarrow \Isom(M)$ is a monomorphism.
	\end{enumerate}
\end{theoremfillinglengths}

Definition~\ref{dfn:symmetrybreak} establishes that a multislope $\psi$ for $M$ is {\em fully generic} if $\psi \in \Psi_{C_M}$ and  {\em symmetry-breaking} if $g(\psi) \neq \psi$ for all non-trivial isometries $g \in \Isom(M)$.

\medskip

In our main application of Theorem~\ref{thm:fillinglengths}, $M$ is one of a handful of hyperbolic manifolds obtained from a partial filling of the exterior of the link $K \cup \calL$.  These are chosen so that a further filling along any multislope in a particular collection $\mathcal{S}_M$ produces a one-cusped manifold with two lens space fillings that is the exterior of a subfamily of knots $K_k(m,r,s,b)$.   The isometry group of each $M$ is calculated with SnapPy.  From this, we conclude that the multislopes in $\mathcal{S}_M$ are symmetry-breaking for most of these $M$.  This is done in Section~\ref{sec:hypandsym}.  Hence, for most of these $M$, we can conclude that the fillings along the fully generic multislopes in $\mathcal{S}_M$ produce asymmetric, one-cusped hyperbolic manifolds with two lens space fillings. Thus the corresponding knots $K_k(m,r,s,b)$ and their surgery duals $K^*_k(m,r,s,b)$ cannot be $(1,1)$--knots.   These results are   the primary content of Theorem~\ref{thm:bulk}, Theorem~\ref{thm:bulk2}, and Addendum~\ref{add:t12685}.

Among the fillings of these $M$ with multislopes in $\mathcal{S}_M$, we locate the DHL manifolds.
However, as we have no estimate on an upper bound for $C_M$ in Theorem~\ref{thm:fillinglengths} and hence no explicit length bound that ensures a multislope is fully generic, our results only establish  the existence of infinitely many examples of asymmetric hyperbolic knots with two lens space fillings, together with finitely many specific, individually calculated, examples.  For instance,  Theorem~\ref{thm:bulk} shows that $K_k(-1,-1,-4+\tfrac1n, b)$ is a knot in a lens space with a lens space filling for all $n,b,k\in\Z$, while Lemma~\ref{lem:bulksymmetry} shows the knot is hyperbolic and asymmetric for $n,b,k\in \Z$ of suitably large magnitude.  However we cannot say how large is ``suitably large''.

Addendum~\ref{add:t12685} gives several $M$ with multislopes $\mathcal{S}_M$ producing one-cusped manifolds with two lens space fillings,  although the multislopes in $\mathcal{S}_M$ are not symmetry-breaking.  A similar argument allows us to conclude that a strong involution is the only non-trivial isometry  of the fillings along the fully generic multislopes in $\mathcal{S}_M$.  Thus other arguments are required to determine whether the corresponding knots $K_k(m,r,s,b)$ or their surgery duals $K^*_k(m,r,s,b)$ are $(1,1)$--knots.

\medskip
Theorem~\ref{thm:bulk} also   adjusts  the constraints on the parameters $m,r,s,b$ so that the knots $K_k(m,r,s,b)$ give surgeries between lens spaces and reducible manifolds or surgeries between two reducible manifolds.  Remark~\ref{rem:bulk2extension} discusses similar modifications to  Theorem~\ref{thm:bulk2}. Theorem~\ref{thm:fillinglengths} may be applied as described above to determine the isometry groups of the exteriors of these knots corresponding to totally generic fillings of the appropriate intermediary filling $M$.  In particular, while Theorem~\ref{thm:bulk} shows that for $b,k\in\Z$ both $K_k(-1,-1,-4,b)$ and $K_k(-1,-2,-3,b)$ are knots in a connected sum of lens spaces with a lens space surgery,  Lemma~\ref{lem:bulkreduciblesymmetry} shows that they are hyperbolic and asymmetric for $b,k\in\Z$ of suitably large magnitude.

\begin{figure}
\includegraphics[height=2.5in]{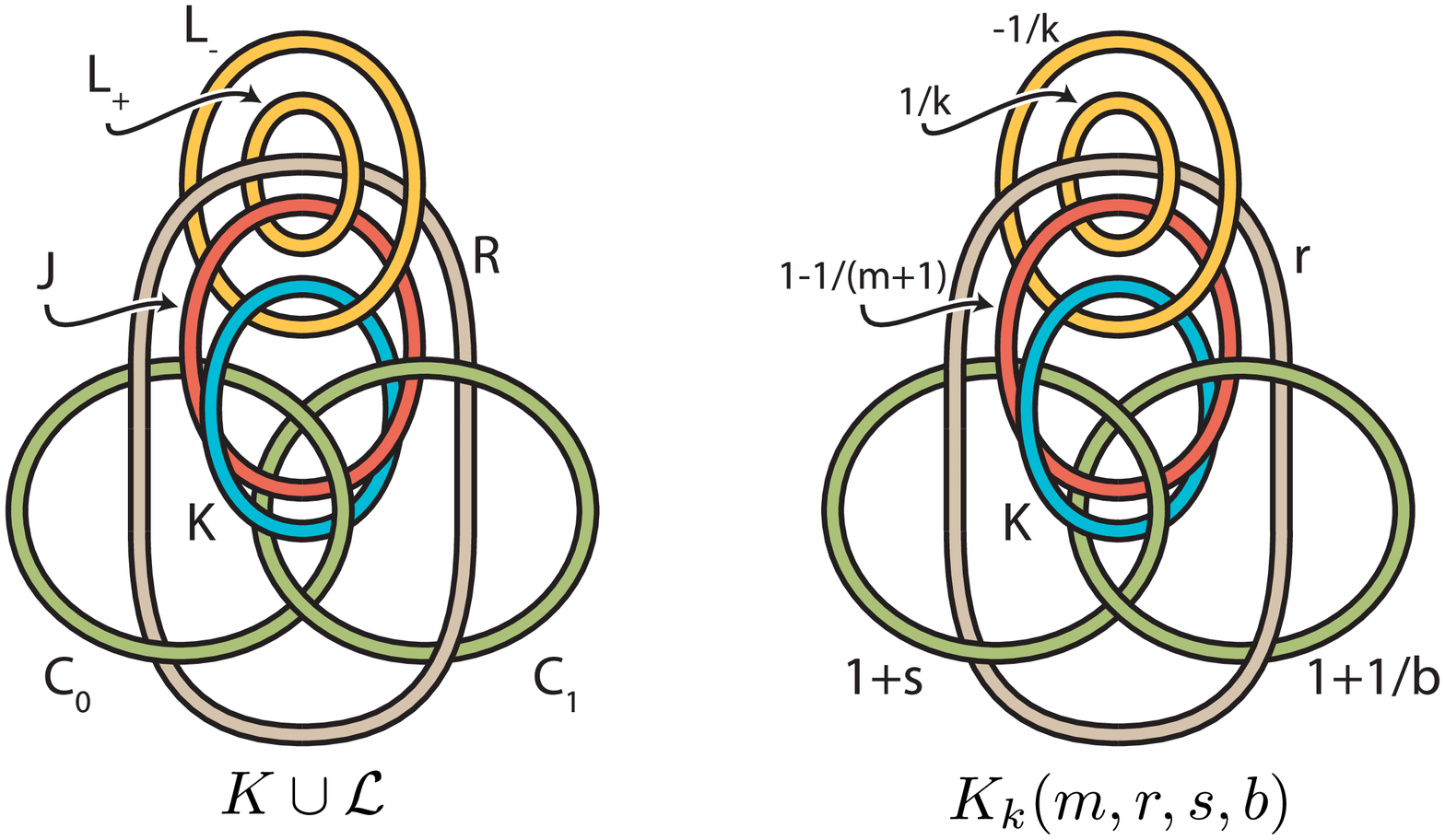}
\caption{ Left presents the link $\calL =  J \cup  R \cup C_0 \cup C_1\cup L_+ \cup L_-$ together with the knot $K$.  Right uses the link $K \cup \calL $ to give a Dehn surgery diagram for the knot $K_k(m,r,s,b)$.}  
\label{fig:surgerylink}
\end{figure}

\begin{figure}
	\includegraphics[height=3.5in]{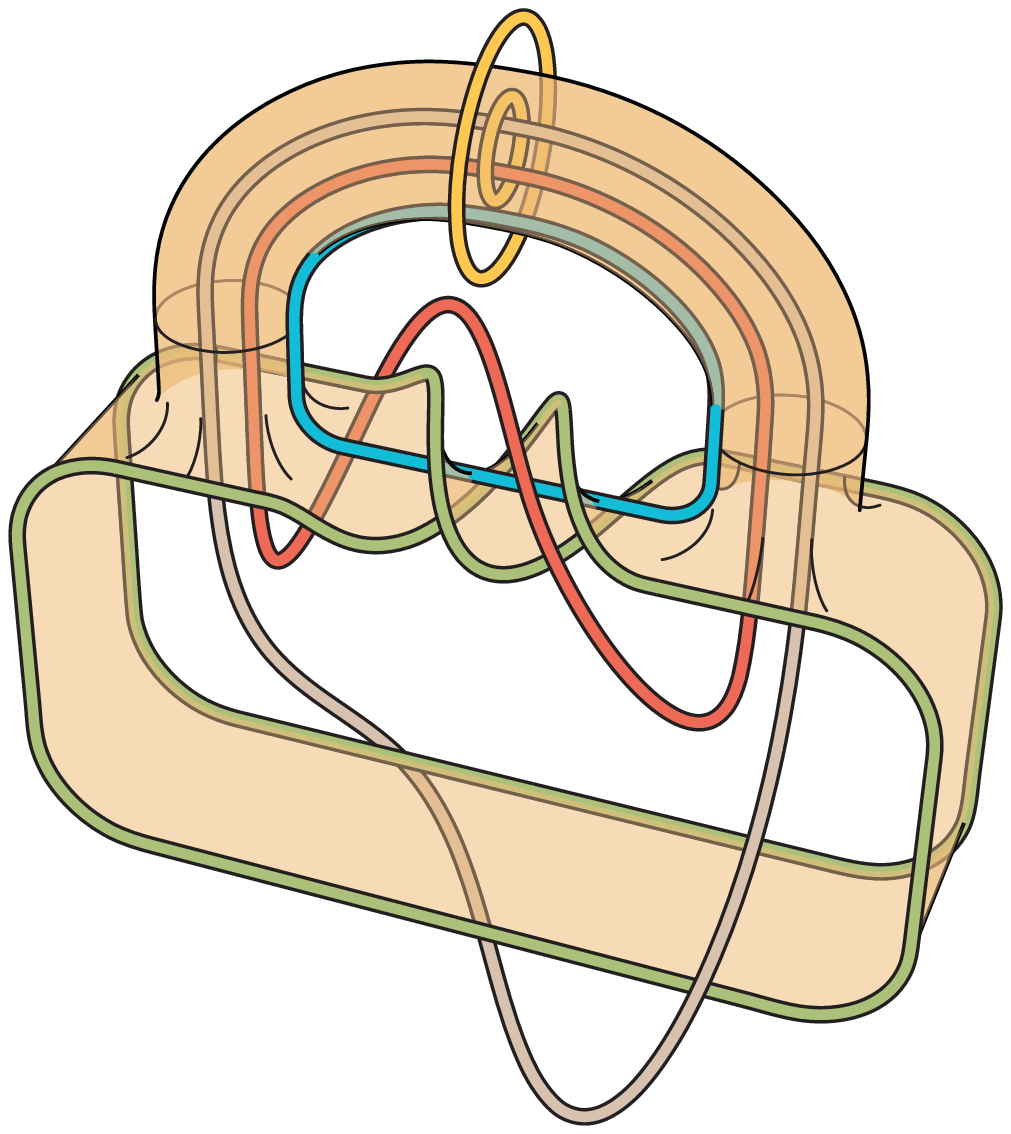}
	\caption{The link $K \cup \calL$ from Figure~\ref{fig:surgerylink} is redrawn here with the  twice-punctured torus Seifert surface $\Sigma$ for $C_0 \cup C_1$.  Observe that $K$ embeds in $\Sigma$ with framing $+1$, $L_+$ and $L_-$ are unknots that cobound an annulus with framing $0$ that intersects $\Sigma$ in a curve that $K$ crosses once, and the curves $J$ and $R$ are disjoint from both $\Sigma$ and this annulus.  }  
	\label{fig:masterlinksurface}
\end{figure}

\subsection{Overview}

Our families of examples are derived from families of strongly invertible, hyperbolic framed knots with a {\em jointly primitive presentations}  either in a cable space or in a Whitehead link exterior for which the strong inversion exchanges the two torus boundary components of the ambient manifold.  The jointly primitive presentation of the knot ensures that the framed surgery produces a cable space.  As cable spaces and Whitehead links have many lens space fillings, selecting such fillings in which the jointly primitive knot becomes a {\em longitudinally} jointly primitive knot ensures that the surgered manifold is a lens space. Further selecting such fillings so that any strong involution of the knot does not extend to the filled manifold then obstructs the knot in the lens space from being  strongly invertible.

Section~\ref{sec:jpp} introduces {\em jointly primitive presentations} of knots and the key Theorem~\ref{thm:jpequivalence} unfolds from this definition.   Section~\ref{sec:questions} then relates jointly primitive presentations  to  other significant surgery presentations and proposes some related questions; we  welcome  further discussion or results that address these.
Section~\ref{sect:symm_and_Dehn_surgery}, and particularly Theorem~\ref{thm:fillinglengths}, clarifies  the hyperbolic geometry of multi-cusped fillings that ensures a ``generic'' filling does not have any unexpected symmetries.  
Section~\ref{sec:families} develops several explicit families of knots in lens spaces exhibiting longitudinally jointly primitive presentations and checks their hyperbolicity and symmetries.  
These results are presented in Theorem~\ref{thm:bulk}, Theorem~\ref{thm:bulk2}, and Addendum~\ref{add:t12685} and represent the culmination of our  goal of placing the DHL manifolds in a natural framework.  With Theorem~\ref{thm:jpequivalence}, they give sense to the DHL manifolds as  recorded in Table~\ref{table:DHL22} and Corollary~\ref{cor:dhl}.

Appendix~\ref{sec:MMM_twice_drilled} notes that ``natural'' examples of hyperbolic manifolds that are not strongly invertible but yet have two lens space fillings  arise from a closer examination of \cite{mattman2006seifert}.  We identify these examples among our families.

We stress that the techniques outlined in this paper should give rise to many more asymmetric hyperbolic manifolds with two lens space fillings.  As a testament to this,  Appendix~\ref{sec:magic} develops more examples of such manifolds using these techniques with fillings of the ``Magic Manifold''.

\subsection{Proof of Corollary \ref{cor:dhl}}

We now summarize the parts of our discussion needed to prove Corollary \ref{cor:dhl}.

\begin{proof}[Proof of Corollary \ref{cor:dhl}]
As indicated in Table \ref{table:DHL22}, each of the 22 DHL manifolds can be obtained via surgery on the link in Figure \ref{fig:surgerylink}. It is now a simple matter of bookkeeping to see that we have shown that the 
relevant surgery parameters yield LJP knots. First, all rows of the table not decorated by  $\dagger$ or $\ddagger$ can be seen as LJP by Theorem \ref{thm:bulk}. 
Four of the remaining rows which are marked with a $\dagger$ are seen as LJP by Theorem \ref{thm:bulk2}. The final case is addressed by Addendum \ref{add:t12685}.
\end{proof}

 \subsection{Acknowledgements}
 
Parts of this work were greatly aided by experimental exploration. The authors wish to thank Marc Culler, Nathan Dunfield, and Matthias Goerner for maintaining SnapPy \cite{snappy},  Benjamin A. Burton and Ryan Budney and William Pettersson for maintaining Regina \cite{regina}, and Frank Swenton for his Kirby Calculator \cite{kirbycalculator}.

This project benefited from the hospitality enjoyed during visits to the University of Melbourne and the Mathematical Sciences Institute at the Australian National University. We thank them for providing this material support to facilitate our work and subsequent writing. KLB would also like to thank CIRGET and Universit\'e du Qu\'ebec \`a Montr\'eal for their support and hospitality during portions of the research and writing.

KLB was partially supported by grants from the Simons Foundation (\#209184 and \#523883 to Kenneth L.\ Baker). NRH was partially supported by grant from the Simons Foundation (\#524123 to Neil R. Hoffman). 

 \begin{table}
 \small
\caption{
	The $22$ DHL manifolds as the knot exteriors $X_k(m,r,s,b)$ and their lens space fillings.  The parameters producing these manifolds are indicated; a few may be obtained from two sets of parameters.
	When unmarked in the final column, the manifold is obtained via Theorem~\ref{thm:bulk}.  When marked with $\dagger$, it is obtained with Theorem~\ref{thm:bulk2}. When marked with $\ddagger$, it is obtained in Addendum~\ref{add:t12685}.}
\label{table:DHL22}
\begin{tabular}[t]{@{}lllrrrrrr@{}}
\toprule
      $X_k(m,r,s,b)$ &    $Y(m,r,s,b)$ &   $Y^*_k(m,r,s,b)$ & $m$ & $r$ & $s$ & $b$ & $k$ \\
\midrule
      $v3372$ &         $L(7, 1)$ &   $L(19, 7)$    &   -1  & -1  &  -3  &  1  &  -2 \\
       	    &                  $L(19, 7)$ & $L(7, 1)$ &  -1  & -2  &  -4  &  -2  &  1 \\
     $t10397$ &      $L(11, 2)$ &   $L(14, 3)$ &     -1  &   -1 & -3 & -2 & 2\\
     $t10448$ &        $L(17, 5)$ &   $L(29, 8)$   & -1  & -2  &  -4  &  1  &  1 \\
     $t11289$ &      $L(11, 2)$ &   $L(26, 7)$    &  -1  &  -1 & -3 & -2 & -2\\
 				&      $L(26, 7)$ &    $L(11, 2)$ &  -1  &  -2  &  -$\tfrac52$  &  -2  &  1 \\
    \midrule 
     $t11581$ &          $L(7, 1)$ &  $L(31, 12)$ &  -1  &   -1  &  -3  &  1  &  2 \\
     $t11780$ &    $L(6,1)$ & $L(23,7)$  &-2 & 0&-4 &1&-1 & $\dagger$\\
     $t11824$ &     $L(19, 4)$ & $L(34, 13)$ & -1  &   -2 & -$\tfrac52$ & 1 & 1  \\
     $t12685$ &          $L(14,3)$ & $L(29,8)$  &  -2&0&-3&-2 & -1  & $\ddagger$\\
     \midrule

 $o9_{34328}$ &      $L(13, 2)$ &  $L(34, 13)$ &  -1  &  -1 & -3 & 2 & -2 \\
						 	&  $L(34, 13)$ & $L(13, 2)$ &   -1  & -2 & -$\tfrac72$ & -2 & 1 \\
 $o9_{35609}$ &          $L(29, 8)$ & $L(50, 19)$ & -1  &  -2 & -$\tfrac72$ & 1 & 1 \\
 $o9_{35746}$ &        $L(17, 3)$ &  $L(41, 12)$ & -1  &  -1 & -3 & -3 & -2 \\
						&   $L(41, 12)$ & $L(17, 3)$ & -1  & -2 & -$\tfrac83$ & -2 & 1 \\
 $o9_{36591}$ &          $L(31, 7)$ & $L(55, 21)$ & -1  &  -2 & -$\tfrac83$ & 1 & 1 \\
\midrule
 $o9_{37290}$ &       $L(31, 12)$ &   $L(19, 4)$ & -1  &  -2 & -4 & -3 & 1 \\
 $o9_{37552}$ &         $L(13, 3)$ &  $L(35, 8)$ &  -1  &   -1  &  -5  &  1  &  -2 \\
 $o9_{38147}$ &        $L(29, 12)$ &  $L(41, 11)$ & -1  &   -2 & -4 & 2 & 1 \\
 $o9_{38375}$ &         $L(17, 3)$ &   $L(29, 8)$   & -1  & -1 & -3 & -3 & 2 \\
			
 \midrule
 $o9_{38845}$ &         $L(18,5)$ & $L(13,2)$  &-2&0&-4&-2&2& $\dagger$\\
 $o9_{39220}$ &        $L(13, 2)$ &  $L(46, 17)$ &  -1  &  -1  &  -3  &  2  &  2 \\
 $o9_{41039}$ &    $L(21,8)$ & $L(13,2)$ &  -3 & 0 & -3 & -2 &   2 & $\dagger$ \\ 
 $o9_{41063}$ &          $L(41, 11)$ &  $L(26, 7)$ &  -1  &  -2 & -$\tfrac52$ & 1 & 1  \\
\midrule 

 $o9_{41329}$ &       $L(34, 9)$ &  $L(49, 18)$ & -1  &  -2 & -$\tfrac52$ & 2 & 1  \\
 $o9_{43248}$ &     $L(18,5)$ &  $L(37,8)$ & -2&0&-4&-2& -1 & $\dagger$\\
\bottomrule
\end{tabular}
\end{table}

\section{Jointly primitive presentations}\label{sec:jpp}

 A collection of simple closed curves $\{K_1, \dots, K_n\}$ in the boundary of a genus $g$ handlebody $H$ with $n\leq g$ is called {\em jointly primitive} if there exists a set of mutually disjoint meridional disks $\{D_1, \dots, D_n\}$ transverse to the curves such that $|K_i \cap D_j| = \delta_{ij}$.  Equivalently, let  $H^* =H[K_1, \dots, K_n]$ be the manifold formed by  attaching $2$--handles to $H$ along the curves $\{K_i\}$.  Then $\{K_1, \dots, K_n\}$ is jointly primitive if  the core arcs $\{K_1^*, \dots, K_n^*\}$ of these $2$--handles are  collectively trivial in $H^*$; the disks $D_i$ become mutually disjoint bridge disks for the arcs $K_i^*$.
 See \cite{Gordon}, \cite{Wu}.

\begin{dfn}
Suppose a $3$--manifold $M$ contains a properly embedded twice-punctured torus $\Sigma$ such that $H =M\setminus\nbhd(\Sigma)$ is a handlebody.  Let  $K$ be a framed knot in $M$ that may be isotoped to lie in $\Sigma$ so that it is non-separating and its framing agrees with the framing by $\Sigma$.  Such an embedding of $K$ in $\Sigma$ is a {\em jointly primitive presentation} of $K$  if the two impressions $K_+$ and $K_-$ in $\bdry \nbhd(\Sigma)$ are a jointly primitive pair of curves in $H$.  A framed knot $K$ is {\em jointly primitive}, or just {\em JP}, if it has a jointly primitive presentation.
\end{dfn}

Recall that a {\em cable space} is a Seifert fibered space over the annulus with (at most) one (possibly degenerate) exceptional fiber.   In a cable space, an annulus that is a union of regular Seifert fibers and that connects the two boundary components is called a {\em spanning annulus}.   Letting $A$ denote the annulus, denote by $A(p/q)$ the cable space obtained by $p/q$--surgery on an interior $S^1$ fiber of the product $A \times S^1$, with surgery description as in Figure~\ref{fig:cablespace}.
In $A(p/q)$, the exceptional fiber has order $|p|$ and is the core of the surgery, and we say the cable space has {\em order} $|p|$. When $|p|>1$, the cable space is the exterior of a torus knot in the solid torus.  When $|p|=1$, the exceptional fiber is actually a regular fiber and the cable space is the thickened torus.  When $p=0$,  the exceptional fiber is degenerate and the cable space is a connected sum of two solid tori (the exterior of the trivial two-component link). 

\begin{figure}
\includegraphics[width=.75\textwidth]{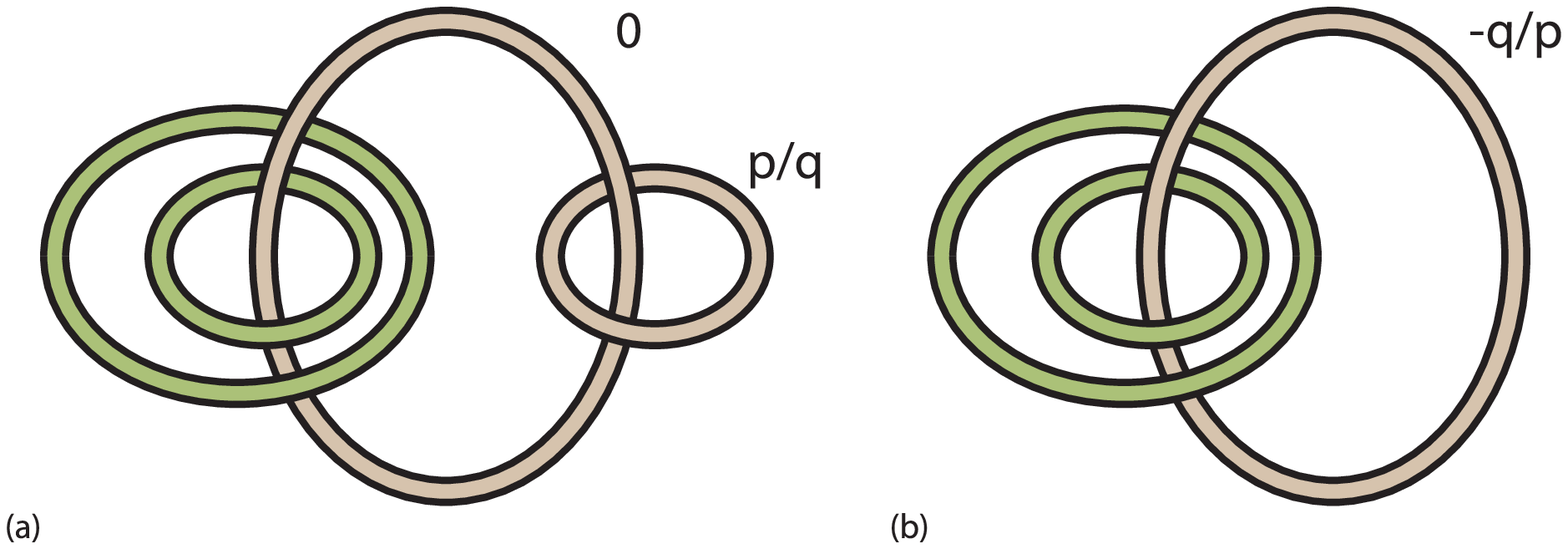}
\caption{(a) Performing the indicated $0$--surgery produces $S^1 \times S^2$ in which the other three components are $S^1$ fibers.  Then performing the $p/q$--surgery makes the exterior of the remaining two copmonents into the cable space $A(p/q)$. (b) A slam dunk produces an alternative surgery diagram of the cable space $A(p/q)$.}
\label{fig:cablespace}
\end{figure}

\begin{remark}\label{rem:cabletorus}
With this definition, the cable space $A(p/q)$ is orientation-preserving homeomorphic to the exterior of the $(p,q')$--torus knot in the solid torus where $qq' \equiv -1 \mod p$.
\end{remark}

\begin{lemma}\label{lem:jptocablespace}
Let $M$ be  a 3-manifold which contains a properly embedded twice-punctured torus $\Sigma$ such that $H =M \setminus N(\Sigma)$ is a handlebody. Assume a framed knot  $K$ in $M$ has a jointly primitive presentation in $\Sigma$.  If $\bdry M$ is two tori, then surface framed surgery on $K$ produces a cable space in which the image of $\Sigma$ is a spanning annulus. 
\end{lemma}

\begin{proof}
Upon performing surgery along $K$, $\Sigma$ becomes an annulus $\Sigma^*$.  This surgery has the effect of gluing $2$--handles to  the genus $3$ handlebody $H=M\setminus \nbhd(\Sigma)$ along each of the impressions $K_\pm$ of $K$.  Hence the result $M^*$ is the solid torus $H^*=H[K_+,K_-]$ glued to itself along the two impressions of the annulus $\Sigma^*$ in $\bdry H^*$.  
As $M$ has two boundary components, so must $M^*$.   
We may orient the core curves of the two impressions of $\Sigma^*$ to be  parallel curves in the torus $\bdry H^*$,  so that gluing  these annuli identifies these oriented curves coherently.

The fibration of the annulus $\Sigma^*$ by circles extends to a Seifert fibration of $M^*$ in which the core curve of the solid torus $H^*$ is the only possible exceptional fiber.  Since $\bdry M$ is two tori, $M^*$ is a cable space and $\Sigma^*$ is its spanning annulus.
\end{proof}

\begin{remark}\label{rem:twistedcablespace}
Let us define a {\em twisted cable space} to be an (orientable) Seifert fibered space over the \mobius band with (at most) one (possibly degenerate) exceptional fiber. A non-separating properly embedded annulus of regular fibers is a {\em twisted spanning annulus}.   
 Lemma~\ref{lem:jptocablespace} and its proof extend immediately to show that if $\bdry M$ is a single torus, then surgery on $K$ produces a twisted cable space in which the image of $\Sigma$ is a twisted spanning annulus.
\end{remark}

Observe that fillings of cable spaces are small Seifert fibered spaces, lens spaces, and connected sums of lens spaces.  
This construction therefore suggests a recipe for creating framed knots $K$ in closed $3$--manifolds $Y$ with surgeries to a filling of a cable space.

Consider a two-component link $C_0 \cup C_1 \subset Y$ whose exterior $M$ contains a twice-punctured torus $\Sigma$ with handlebody complement, and let $K$ be any framed knot with a jointly primitive presentation in $\Sigma$.  By Lemma~\ref{lem:jptocablespace}, surgery on $K$ transforms $M$ into a cable space $M^*$.  Hence surgery on $K$ creates the small Seifert fibered space which is the filling of the cable space $M^*$ by the solid tori $\nbhd(C_0)$ and $\nbhd(C_1)$.  

Assume for the moment that the cable space has order at least $2$. 
Since $\Sigma$ becomes a spanning annulus, a lens space is produced whenever $\bdry \Sigma$ meets one of $\bdry \nbhd(C_i)$ in a longitudinal curve.   Similarly, a connected sum of lens spaces is produced when $\bdry \Sigma$ meets one of $\bdry \nbhd(C_i)$ in a meridional curve.  If $\bdry \Sigma$ meets both $\bdry \nbhd(C_0)$ and $\bdry \nbhd(C_1)$ in non-longitudinal, non-meridional curves, then the result is a small Seifert fibered space that is neither a lens space nor a connected sum.  We refer to these three situations as conferring {\em longitudinal}, {\em meridional}, or {\em rational} jointly primitive presentations on the framed knot $K$ in $Y$.  For short, we say such framed knots are {\em LJP}, {\em MJP}, or {\em RJP}, respectively.

\begin{remark}
$\quad$ 
\begin{enumerate}
\item In this article, we focus primarily on the case of LJP framed knots, though examples and results concerning meridional jointly primitive presentations also arise naturally.   

\item In light of Remark~\ref{rem:twistedcablespace}, one may similarly consider jointly primitive framed knots $K$ in a twice-punctured torus rational Seifert surface $\Sigma$ for a knot $C$ in a closed $3$--manifold $Y$. By joining the two components of $\bdry \Sigma$ (while regarding them in $\bdry \nbhd(C)$), such knots $K$ have surgeries to manifolds containing Klein bottles.  In particular, if each component of $\bdry \Sigma$ were a longitude of $C$, $\Sigma \cup C$ would be a $1$--sided Heegaard splitting of $Y$ and $K$ would have a surgery to a prism manifold.  We do not pursue this case further here. 

\item One could also define a jointly primitive presentation of a framed knot with respect to a twice-punctured Klein bottle.  We do not address this either, beyond Remark~\ref{rem:2bridge} and a mention in Section~\ref{sec:mjp}.  \end{enumerate}
\end{remark}

\subsection{LJP knots and $(1,2)$--knots are surgery dual.}
Recall that a lens space decomposes along a torus into two solid tori $V$ and $V'$.  A knot in the lens space is a {\em $(1,n)$--knot} if it can be isotoped to meet each of $V$ and $V'$ in a collection of $n$ mutually trivial arcs.  More generally, if a manifold splits along a surface into two handlebodies of genus $g$ (i.e., it has a genus $g$ Heegaard splitting), then a knot in the manifold is a {\em $(g,b)$--knot} if it can be isotoped to meet each handlebody in a collection of $b$ mutually trivial arcs.  We say the {\em genus $g$ bridge number} of such knot is $b$ if it is not a $(g,b-1)$--knot with respect to any genus $g$ Heegaard splitting.

\begin{theorem}\label{thm:jpequivalence}
A framed knot $K$ has a longitudinal jointly primitive presentation if and only if its framed surgery dual $K^*$ is a $(1,2)$--knot in a lens space.
\end{theorem}

\begin{proof} 
Throughout, let $Y^*$ be the lens space containing $K^*$ which is surgery dual to some framed knot $K$ in a $3$--manifold $Y$.

\smallskip
\noindent{\em If $K$ is LJP, then $K^*$ is a $(1,2)$--knot:}

Assuming $K \subset Y$ is LJP, there is a two-component link $C_0 \cup C_1 \subset Y$ and a twice-punctured torus $\Sigma$ properly embedded in $M = Y-\nbhd(C_0 \cup C_1)$ such that 
\begin{enumerate}
\item $K$ is a non-separating curve in $\Sigma$ whose integral framing is given by $\Sigma$;
\item $M -\nbhd(\Sigma)$ is a genus $3$ handlebody $H$;
\item the two impressions $K_\pm$ of $K$ in $\bdry H$ are jointly primitive in $H$; and
\item $\sigma_1=\bdry \Sigma \cap \bdry \nbhd(C_1)$ is a longitude of $C_1$.
\end{enumerate}
To perform framed surgery on $K$, first regard $\nbhd(\Sigma)$ as $\Sigma \times [-1,1]$ and $\nbhd(K)$ as $\big(K\times[-1,1]\big) \times  [-1,1]$ where the first interval factor determines an annular neighborhood of $K$ in $\Sigma$.  Then $\bdry \nbhd(K)$ naturally decomposes into four annuli, two of which are horizontal annular neighborhoods of the impressions $K_\pm$ in $\bdry H$ while the other two are the vertical annuli $\big(K \times \{\pm1\}\big) \times [-1,1]$.  Performing $\Sigma$--framed surgery on $K$ is then equivalent to attaching four $2$--handles to $Y-\nbhd(K)$ along these annuli. Furthermore, the union of the cocores of these four $2$--handles is the surgery dual knot $K^*$ in the resulting manifold $Y^*$. 
We will show that this structure naturally gives a two-bridge presentation of $K^*$ with respect to a genus $1$ Heegaard splitting of $Y^*$.

Decompose $Y-\nbhd(K)$ as the handlebody $H$ and the manifold $H' = \nbhd(C_0 \cup C_1) \cup \nbhd (\Sigma')$, where $\Sigma' = \Sigma - \nbhd(K)$. We may regard $\nbhd(\Sigma')$ as $\Sigma' \times [-1,1]$. This induces a decomposition of $Y^*$ into $V=H[K_+,K_-]$ and $V'=H'[K'_+,K'_-]$ where $K'_\pm=\bdry \Sigma'-\bdry \Sigma$ are the cores of the vertical annuli.

The jointly primitive presentation of $K$ in $\Sigma$ implies that the curves $K_\pm$ are jointly primitive in $H$.  Hence, attaching $2$--handles to $H$ along the impressions $K_\pm$ yields a solid torus $V = H[K_+,K_-]$ in which the cocores of these two $2$--handles form a  pair of trivial arcs. 

Since $K$ is a non-separating curve in $\Sigma$, $\Sigma'$ is a $4$--punctured sphere.  Hence $H'= \nbhd(C_0 \cup C_1) \cup \nbhd (\Sigma')$ is a genus $3$ handlebody.
Again because $K$ is a non-separating curve in $\Sigma$, there is a properly embedded arc $a$ in $\Sigma$ with $\bdry a \subset \sigma_1$ that crosses $K$ once.  Restricting to $\Sigma'$, $a$ splits into two arcs $a_\pm$ connecting $K'_\pm$ to $\sigma_1$.  Then in $\nbhd(\Sigma')$, these two arcs give rise to properly embedded disks $D'_\pm$ such that $|K'_i \cap D'_j| = \delta_{ij}$ for $i,j \in \{ +, -\}$.  Furthermore, since these disks each cross $\sigma_1 \times \{0\}$ once and $\sigma_1$ is a longitude of $C_1$, they may be extended through $\nbhd(C_1)$ to meridional disks of $H'$ without altering how they intersect $K'_\pm$.  Thus the pair of curves $K'_\pm$ are jointly primitive in the handlebody $H'$.  Attaching $2$--handles to $H'$ along the curves $K'_\pm$ therefore  yields a solid torus $V' = H[K'_+,K'_-]$ in which the cocores of these two $2$--handles form a  pair of trivial arcs. 

Thus $Y^*$ is a lens space and $K^*$ is a $(1,2)$--knot.

\medskip
\noindent
{\em If $K^*$ is a $(1,2)$--knot, then $K$ is LJP:}

Now assume $K^*$ is a knot with a $(1,2)$--presentation in the manifold $Y^*$.  That is, $Y^*$ admits a genus one Heegaard splitting as the union of two solid tori $V \cup V'$ with the property that $K^*$ intersects each of $V$ and $V'$ in a trivial pair of arcs which we denote as $K^*_a,K^*_c$ and $K^*_b,K^*_d$, respectively.  

Choose an orientation of $K^*$ and let these arcs inherit this orientation.   Since $K^*_a,K^*_c$ is a trivial pair of arcs in $V$, there is a disjoint pair of meridional disks $D_a, D_c$ of $V$ that contain the arcs $K^*_a,K^*_c$.  We may now take a core curve $C_0$ of $V$ that intersects each of $D_a$ and $D_c$ once.  Orient $C_0$ and the disks  so that both intersections are  positive.   Furthermore, we may isotope $C_0$, maintaining transversality to $D_a$ and $D_c$, so that its intersection with these disks is to the left of $K^*_a$ and to the right of $K^*_c$. Thus $K^*_a$ and $K^*_c$ cut off subdisks $E_a\subset D_a$ and $E_c\subset D_c$ that are disjoint from $C_0$ and the orientation on $K^*_a$ is opposite the induced orientation on $\bdry E_a$ while the orientation on $K^*_c$ agrees with the induced orientation on $\bdry E_c$.  See Figure~\ref{fig:21toLJP}. Next, choose an essential simple closed curve $C_1$ in $\bdry V$ that intersects $D_a$ and $D_c$ coherently and intersects each of $E_a$ and $E_c$ once.  Then $C_1$ is a cable of $C_0$ in $V$ and a spanning annulus $A$  intersects each of $K^*_a$ and $K^*_c$ once.
   Observe that $V$ may be regarded as a neighborhood of $C_0 \cup A \cup C_1$. 
      
\begin{figure}
\includegraphics[height=2in]{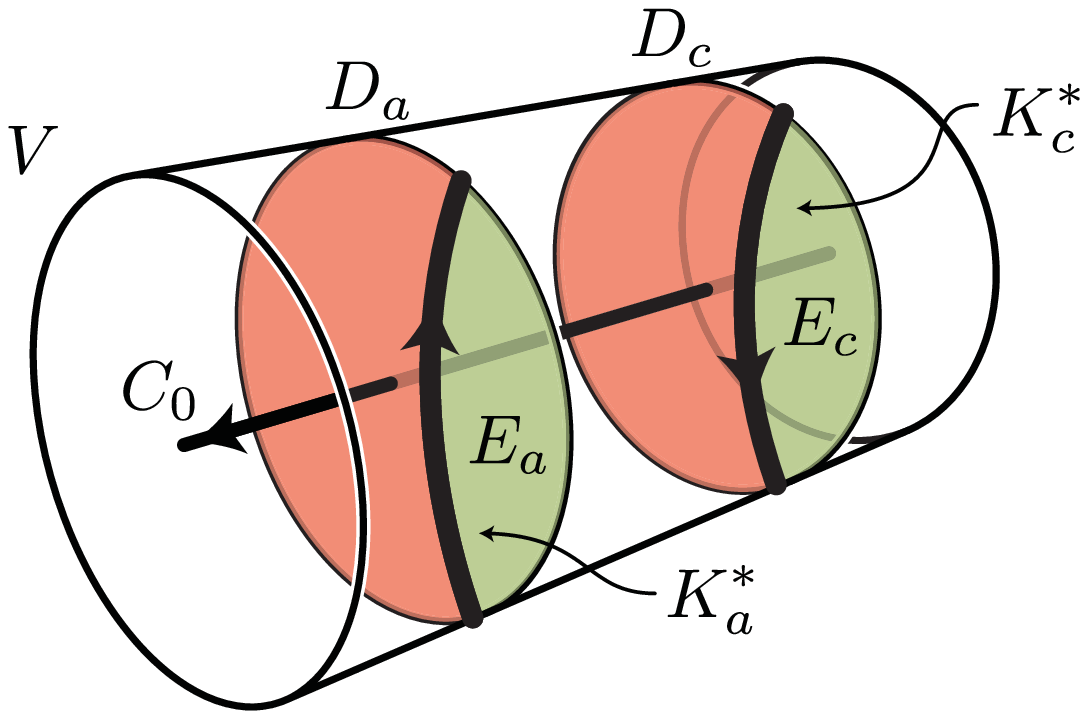}
\caption{The curve $C_0$ is chosen to intersect the disks $D_a$ and $D_c$ disjoint from the subdisks $E_a$ and $E_c$ as shown.
\label{fig:21toLJP}}
\end{figure}
   
  Since $K^*_b$ and $K^*_d$ are a trivial pair of arcs in $V'$, they have disjoint bridge disks $E_b$ and $E_d$ in $V'$; that is, $E_b$ and $E_d$ are a disjoint pair of embedded disks in $V'$ such that $\bdry E_i$ is the union of $K^*_i$ and an arc in $\bdry V'$  for  $i=b,d$.  Thus, in the genus $3$ handlebody $H=V'-\nbhd(K^*_b \cup K^*_d)$, the disks $E_b$ and $E_d$ certify the core curves $K_b$ and $K_d$ of the annuli $\bdry \nbhd(K^*_b) \cap \bdry H$ and $\bdry \nbhd(K^*_d) \cap \bdry H$ as jointly primitive.  

By construction, $K^*$ intersects $A$ geometrically twice but algebraically zero times.  Thus $A$ may be tubed along either arc of $K^*-A$ to form a twice-punctured torus $\Sigma$.  Then after any integral surgery on $K^*$ to a manifold $Y$, the surgery dual curve $K$ may be viewed as lying in $\Sigma$  so that $Y-\nbhd(C_0 \cup \Sigma \cup C_1) = H$ and the impressions of $K$ in $\bdry \nbhd(\Sigma)$ are the curves $K_b$ and $K_d$ in $\bdry H$.  Hence $K$ has a longitudinally jointly primitive presentation in $\Sigma \subset Y$.
\end{proof}

\begin{remark}\label{rem:2bridge}
At the end of the above proof, if $C_0$ were instead isotoped to intersect $D_c$ to the left of $K^*_c$, the above construction would result in $K^*$ intersecting $A$  both algebraically and geometrically twice.  Then tubing $A$ would form a twice-punctured Klein bottle $\Sigma$, and the surgery dual knot $K$ would have a longitudinally jointly primitive presentation in this $\Sigma \subset Y$ instead.  Indeed, the above proof easily extends to show that a framed knot $K$ is longitudinal jointly primitive with respect to a twice-punctured Klein bottle if and only if it is longitudinal jointly primitive with respect to a twice-punctured torus.
\end{remark}

\begin{cor}\label{cor:tn}
If a framed knot $K$ in a closed $3$--manifold $Y$ has a jointly primitive presentation, then the tunnel number of $K$ is at most $3$ and the Heegaard genus of $Y$ is at most $4$.  If $K$ has a longitudinally jointly primitive presentation, then its tunnel number is at most $2$ and the Heegaard genus of $Y$ is at most $3$.
\end{cor}
\begin{proof}
Since the Heegaard genus of a manifold containing a tunnel number $n$ knot is at most $n+1$, we  need only bound the tunnel numbers of these knots.

Suppose $K$ has a jointly primitive presentation with respect to $\Sigma$.  Since $\Sigma - K$ may be cut into a disk $D$ by a pair of arcs running from $K$ to each component of $\bdry \Sigma$  together with a third arc from $K$ to itself, the union of these arcs with $\partial \Sigma=C_0 \cup C_1$ form a $3$ tunnel system for $K$.  The exterior of $K$ and these tunnels is a genus $4$ handlebody that may be viewed as the genus $3$ handlebody $H = Y-\nbhd(C_0 \cup \Sigma \cup C_1)$ together with an extra $1$--handle whose cocore is the disk $D$.  

Since a knot with a $(g,b)$--presentation has tunnel number at most $g+b-1$, a knot with a $(1,2)$--presentation has tunnel number at most $2$.  If $K$ has a longitudinally jointly primitive presentation, then its surgery dual has a $(1,2)$--presentation by Theorem~\ref{thm:jpequivalence}.   Hence $K$ has tunnel number at most $2$.
\end{proof}

\subsection{Fibered jointly primitive presentations}
The proof of Theorem~\ref{thm:jpequivalence}  shows that an integral surgery dual to a $(1,2)$--knot has infinitely many LJP presentations,  distinguished by the order of wrapping of $\bdry \Sigma$ about $C_0$. 
Indeed, one may always find an LJP presentation in which $\Sigma$ is longitudinal on each of $C_0$ and $C_1$; in particular, $\Sigma$ will be a Seifert surface for the link $C_0 \cup C_1$.     It seems desirable to identify a further feature that gives preference to certain jointly primitive presentations over others.

\begin{dfn}
Suppose  $K$ is a knot with a jointly primitive presentation in a twice-punctured torus $\Sigma$ properly embedded in the exterior of a link $C_0 \cup C_1\subset Y$.  The jointly primitive presentation  is {\em fibered} if   $\Sigma$ is a fiber in a fibration of $Y-\nbhd(C_0 \cup C_1)$. 
\end{dfn}  

\begin{remark}
When  $\Sigma$ is a fiber, the handlebody $H=Y-\nbhd(C_0 \cup C_1)-\nbhd(\Sigma)$ is a product $\Sigma \times [-1,1]$ in which the impressions of $K\subset \Sigma$ are the curves $K_\pm = K\times \{\pm1\}$.   As such, each curve $K_\pm$  is individually already a primitive curve in $H$; hence $H[K_+, K_-]$ is a solid torus only if the pair is jointly primitive. 
\end{remark}

As will be shown, the knots of Theorem~\ref{thm:bulkcable}, Theorem~\ref{thm:bulk2}, and Addendum~\ref{add:t12685} 
are all fibered LJP. 
So it's reasonable to ask if Theorem~\ref{thm:jpequivalence} can be improved to a statement about fibered longitudinal jointly primitive presentations.

\begin{question}
If a framed knot $K \subset Y$ has a longitudinally jointly primitive presentation, then does it also have a fibered longitudinally jointly primitive 
 presentation?  
\end{question}

Indeed, many of the LJP knots in lens spaces obtained in Appendix~\ref{sec:magic} are not given as fibered LJP knots and do not obviously also have a fibered LJP presentation.  Furthermore, doubly primitive knots are LJP as we will see in Corollary~\ref{cor:doublyprimitive}, but it is not immediately clear whether or not they have fibered LJP presentations.

\subsection{Connections and questions}\label{sec:questions}

The notion of a jointly primitive presentation introduced in this paper has close connections to other objects in the literature.  This subsection notes a few important relationships and poses some further questions. 

\subsubsection{Meridional jointly primitive presentations}\label{sec:mjp}
The reader may notice that a  meridional jointly primitive presentation of a framed knot $K$ is simply a knot $K$ realized as a  framed, non-separating curve in a once-punctured torus (or Klein bottle  $\hat{\Sigma}$) that is (rationally) bounded by another knot $C_0$.  
By the same philosophy as Lemma~\ref{lem:jptocablespace}, surgery on the framed knot $K$ transforms the 
surface $\hat{\Sigma}$ into a disk $D$ so that $\nbhd(C_0 \cup D)$ is a once-punctured lens space.

In another article in progress \cite{baker-newred}, the first author exploits this construction to recover and extend all examples found in the literature of hyperbolic manifolds with two reducible surgeries.  Applying similar techniques, he further extends the examples of reducible surgeries on hyperbolic knots in lens spaces. 

\subsubsection{Doubly primitive knots}\label{sec:doubly_primitive_knots}
A framed knot $K$ in a closed $3$--manifold $Y$ has a {\em doubly primitive presentation} if  $K$ may be isotoped to embed in a genus $2$ Heegaard surface 
so that $K$ is primitive in each of the two handlebodies \cite{berge2018some}.   
The Berge Conjecture asserts that any framed knot in $S^3$ with surgery to a lens space has a doubly primitive presentation \cite{berge2018some, gordon1998dehn}. A similar conjecture proposes that any framed knot in $S^1 \times S^2$ with surgery to a lens space has a doubly primitive presentation \cite{BBLS1xS2}, \cite{Greene2013Realization}.
Since doubly primitive knots are surgery dual to $(1,1)$--knots which stabilize to have $(1,2)$--presentations,
Theorem~\ref{thm:jpequivalence} implies that each doubly primitive knot has a longitudinally jointly primitive presentation.   

\begin{cor}\label{cor:doublyprimitive}
If a framed knot $K$ in a closed $3$--manifold $Y$ has a doubly primitive presentation, then it admits a longitudinally jointly primitive  presentation. \qed
\end{cor}

Given this observation, there are a couple of natural questions exploring this relationship. 

\begin{question}\label{ques:lensLJP}
If a framed knot in $S^3$ or $S^1 \times S^2$ has a lens space surgery, then is it LJP?
\end{question}

\begin{question}
If a framed knot in $S^3$ or $S^1 \times S^2$ is LJP, then is it also doubly primitive?
\end{question}

\begin{question}
If framed knots $K, K^*$ in lens spaces $Y,Y^*$ are surgery dual, then must either $K$ or $K^*$ be LJP?  Equivalently, must either $K$ or $K^*$ have a $(1,2)$--presentation?
\end{question}

Since Berge knots (the doubly primitive knots in $S^3$ \cite{berge2018some, Greene2013Realization}) are LJP by Corollary~\ref{cor:doublyprimitive}, we may make a few inferences.

\begin{remark}
There are Berge knot complements with arbitrarily large volume \cite{baker2008large}. Therefore, it would be impossible for any single link, let along the link in Figure \ref{fig:masterlinksurface}, to be the common ancestor of all LJP knots in a given lens space. 
\end{remark}

\begin{remark}
There are Berge knots in $S^3$ whose genus $1$ bridge number is arbitrarily large (see e.g., \cite[Theorem~1.3]{bowmantaylorzupan} and \cite[Theorem~1.3]{BakerKobayashiRieck}). Hence Theorem~\ref{thm:jpequivalence} implies that, while they are LJP knots themselves, these Berge knots of large genus $1$ bridge number are not surgery dual to LJP knots in lens space.  

Indeed, \cite{BGLintegral} can be used to show that ``most'' of the LJP knots in lens spaces obtained in Theorem~\ref{thm:bulk} have genus $1$ bridge number greater than $2$. Let us use the surgery descriptions and notation of some knots and manifolds described in Section~\ref{sec:surgerydiagram}.  For example, since the exterior of $L_+$ with the surgery duals to $C_0 \cup C_1$ in $Y(-1,r,s,b)$ for $r\in \Z$ can be seen through surgery calculus to be the hyperbolic ``magic manifold'', Thurston's Hyperbolic Dehn Surgery Theorem implies that the image of $L_+$ in $Y(-1,-1,-3+1/n,b)$ is hyperbolic for integers $n,b$ of suitably large magnitude.  Hence the annulus cobounded by $L_+$ and $L_-$ cannot lie in the Heegaard torus of the lens space and the work of \cite{BGLintegral} applies.
\end{remark}

\subsubsection{LJP knots in the Poincare homology sphere}\label{PHS}

In \cite{baker-PHSlens}, longitudinally jointly primitive presentations were used to create framed knots in the Poincare homology sphere  without doubly primitive presentations that nevertheless admit lens space surgeries. The Whitehead link is a fibered link with a twice-punctured torus fiber.  With respect to the framing of the Whitehead link by this fiber, $-1$ surgery on both components produces the Poincare homology sphere.  

In Section~\ref{sec:surgerydiagram} we give surgery descriptions of some manifolds and knots.  Using that notation, the manifold $M(-1,+1)$ is the exterior of the Whitehead link and $Y(-1,+1,-1,-1)$ is the Poincare homology sphere filling of it.
It turns out that there are (at least) two families of jointly primitive knots in $M(-1,+1)$ that induce families of LJP knots in the Poincare homology sphere.  One family may be described as $K_k(-1,+1,-1,-1)$, and it may be shown that these correspond to a subfamily of Hedden's knots \cite{Hedden, Baker-almostsimple}.  The 
other, which corresponds to the family in \cite{baker-PHSlens}, doesn't exactly fit into the $K_k(m,r,s,b)$ structure. Nonetheless, it may be described by surgery on the link $\calL \cup K$ of Figure~\ref{fig:surgerylink} as the images of $K$ 
in 
$S^3_\calL(1/2,0,2,2,1/k,-1/(k+1))$ 
for $k \in\Z$, up to homeomorphism.  We note that the trivial filling of $K$ here then produces the mirror image of $Y(-1,+1,-1,-1)$.  

In line with Question~\ref{ques:lensLJP} we also ask the following. 
\begin{question}\label{q:PHS}
	If a framed knot in the Poincare homology sphere has a lens space surgery, then is it LJP? 
\end{question}


\section{Symmetries and Dehn surgery}\label{sect:symm_and_Dehn_surgery}

A manifold $M$ is \emph{asymmetric} if its only self-homeorphisms are isotopic to the identity.   This section outlines a method for obtaining asymmetric manifolds as fillings of cusped hyperbolic manifolds.

For hyperbolic $3$--manifolds, asymmetry is equivalent to having trivial isometry group, which in the case of cusped manifolds may be checked using the Epstein-Penner canonical cell decomposition (see \cite{weeks1993convex}).   More generally, methods to ensure asymmetry are well known to  experts and are discussed in the one-cusped case in \cite[\S2]{DHL} (see also \cite[Theorem 5.2]{auckly2014two}).  Here, we devote  additional care to the case of multi-cusped manifolds in order to precisely establish our notions of  \emph{generic} and \emph{symmetry-breaking} fillings. (See also \cite[\S4]{bakerluecke-asymmetricLspaceknots}.)   

\subsection{Isometry groups of fillings} \label{sec:isometrygroupsfillings}

For $m \geq 1$, 
let $\bar{M}$ be a closed orientable $3$--manifold whose boundary consists of  tori $T_1, \dots, T_m$.  
For each $j$, fix a basis for $H_1(T_j; \Z)$ so that a slope (an isotopy class of essential simple closed curve) in $T_j$ may be expressed as a pair of relatively prime integers $\psi_j = (p_j,q_j)$.  Since orientation of a slope is irrelevant, we may also express the slope as the rational number $p_j/q_j$ allowing $1/0$.   We further allow ``empty'' slopes which are expressed as the pair $(\infty,\infty)$ or as the symbol $*$. 
A {\em multislope}  for $\bar{M}$ is an $m$--tuple $\psi = (\psi_1, \dots, \psi_m)$ of slopes in $\bdry \bar{M}$.  The set of all multislopes is $\Psi$.  The Dehn filling $\bar{M}(\psi)$ of $\bar{M}$ along the multislope $\psi$ is obtained by attaching  solid tori to $\bdry \bar{M}$ so that the meridian of the solid torus attached to $T_j$ is identified with the slope $\psi_j$.  If $\psi_j$ is the empty slope, then $T_j$ is left unfilled.  The core curves of the attached solid tori in $\bar{M}(\psi)$ are the {\em cores} of the filling.

Let $M$ be the $m$--cusped orientable hyperbolic $3$--manifold that is the interior of $\bar{M}$, where we identify $\bdry \bar{M}$ with $\bdry M$ and  fillings $\bar{M}(\psi)$ with $M(\psi)$.
Following the discussion surrounding \cite[Equations (26) p1068 and (37) p1076]{hodgsonKerckhoff2008}, (but using different notation),  the \emph{normalized length} $|\psi_{j}|$ of a slope $\psi_j$ is the translation distance of the appropriate holonomy representative of $\psi_{j}$ in the horoball neighborhood of area $1$, and  the {\em normalized length}  $|\psi|$ of a multislope $\psi= (\psi_1, \dots, \psi_m)$ is defined by 
  \[ \frac{1}{|\psi|^2} = \sum_j \frac{1}{|\psi_{j}|^2}.\]
  The normalized length of the empty slope is defined to be $\infty$. Consequently, the normalized length of a multislope is finite unless it is the empty multislope $\psi^\infty = \big((\infty,\infty), (\infty,\infty),  \dots, (\infty,\infty)\big)$.
Let $\Psi_N \subset \Psi$ be the subset of multislopes in $\bdry M$ whose component slopes each have normalized length greater than $N$.  That is, 
\[ \Psi_N = \{ \psi = (\psi_1, \dots, \psi_m) \colon |\psi_j| > N \mbox{ for all } j=1, \dots, m \}.\]
Note that for each $N$, the complement of $\Psi_N$ is finite and if $\psi \in \Psi_N$, then $|\psi|>N$.

Consider a hyperbolic $3$--manifold $\widehat{M}$ containing a geodesic link $L$ so that $M=\widehat{M}-L$.  Let $I_L$ be the group of isometries of $N$ which preserve $L$ setwise.  In \cite[Section 5]{kojima}, Kojima describes how restricting the domain of an isometry $f \in I_L$ to $M=\widehat{M}-L$ induces an isometry $rest(f)$ of $M$.  Then \cite[Lemma 5]{kojima} shows that this restriction map is a monomorphism. 

\begin{theorem}\label{thm:fillinglengths}
	Let $M$ be an $m$--cusped orientable hyperbolic $3$--manifold. Set $C= 7.5832$.  
	\begin{enumerate}
		\item For each $\psi \in \Psi_{C}$, $M(\psi)$ is a hyperbolic $3$--manifold in which the cores of the filling are mutually disjoint simple closed geodesics.
		\item There is a constant $C_M>C$ depending on $M$ such that for each $\psi \in \Psi_{C_M}$, the cores of the filling constitute the shortest geodesics in $M(\psi)$ so that Kojima's restriction map $rest$ is a monomorphism $Isom(M(\psi)) \hookrightarrow Isom(M)$.
	\end{enumerate}
\end{theorem}

\begin{remark}
	The manifold-dependent constant $C_M$ in Theorem~\ref{thm:fillinglengths}(2) is necessary.  For example, consider a hyperbolic $3$--manifold $N$ with an isometry that exchanges two disjoint simple geodesics $L_1$ and $L_2$.  Then this isometry of $N$ does not restrict to an isometry of $M=N-L_1$, the complement of one of the geodesics. Hence $Isom(M)$ is not a subgroup of $Isom(N)$.
\end{remark}

\begin{proof}
	\cite[Theorem 1.2]{hodgsonKerckhoff2008} makes Thurston's Hyperbolic Dehn Surgery Theorem \cite[Theorem 5.8.2]{ThurstonNotes} explicit and states that any $\psi \in \Psi_C$ is in the hyperbolic Dehn surgery space for $M$.
	This implies that $M(\psi)$ is a hyperbolic manifold and the core curves of the filling are mutually disjoint geodesics; see \cite{hodgsonKerckhoff2008}.
	
	For $\psi \in \Psi_C$, \cite[Theorem 5.12]{hodgsonKerckhoff2008} bounds the difference between the volumes of $M$ and $M(\psi)$ in terms of $|\psi|$.  By \cite[Theorem 1B]{neumannzagier}, this translates to a bound on the sum of the lengths of the geodesic cores of the fillings. In particular,  as $|\psi|$ goes to $\infty$, the total length of the geodesic cores of the fillings goes to $0$.  
	
	We now partition $\Psi$ according to which components of a multislope are the empty slope.  For each subset $J$ of $\{1, \dots, m\}$,  let $\Psi^J$ be the subset of multislopes $\psi=(\psi_1, \dots, \psi_m)$ in $\Psi$ such that $\psi_i =(\infty,\infty)$ if and only if $i \not \in J$.  (In particular, note that (a) $J$ determines the set of cusps of $M$ that are filled by a multislope $\psi \in \Psi^J$ and (b) $\Psi^\emptyset$ is the singleton set consisting of the empty multislope $\psi^\infty$ for which $M(\psi^\infty) = M$.)  Then let $\Psi^J_N = \Psi^J \cap \Psi_N$.  Since $\Psi_N$ has finite complement in $\Psi$ for all numbers $N$, it follows that the complement of $\Psi^J_N$ is finite in $\Psi^J$.
	
	For each $J \neq \emptyset$, we may order the multislopes in $\Psi^J_C$ as $\{\psi^i\}$ so that for each $N$ we have $\psi^i \in \Psi^J_N$ for all but a finite number of initial terms of the sequence.  In particular, this means that $\psi^i$  converges to the empty multislope $\psi^\infty$; more precisely, the sequence $\{(|\psi_1^i|, \dots, |\psi_m^i|)\}$  converges to $(\infty, \dots, \infty)$ and the corresponding sequence of filled manifolds $M(\psi^i)$ limits to the unfilled manifold $M=M(\psi^\infty)$.

	Then \cite[Theorem E.2.4]{BP} shows that for suitably small $\epsilon >0$ (depending on $M$ and $J$),  there exists a sufficiently large  constant $C^J_M>C$ so that if $\psi\in\Psi_{C^J_M}$, then (a) the total length of core geodesics of the filling $M(\psi)$ is less than $\epsilon$, and (b) the $\epsilon$--thin part of $M(\psi)$ consists of the tubular neighborhoods of the core geodesics and cusp neighborhoods of the unfilled cusps.  By definition, any  geodesic of length less than $\epsilon$ must be in the $\epsilon$--thin part of $M(\psi)$ (e.g., \cite[Section D.1]{BP}).  Furthermore, by \cite[Theorem D.3.11]{BP} these core curves are the unique geodesics in these tubular neighborhoods.  Thus, the cores of the fillings are the unique simple geodesics in $M(\psi)$ of length less than $\epsilon$.	Therefore any isometry of $M(\psi)$ must fix the link of core geodesics of the $\psi$--filling.  Thus \cite[Lemma 5]{kojima} gives the natural monomorphism $rest$ of $Isom(M(\psi))$ into $Isom(M)$.
	
	Finally, set $C_M$ to be the maximum of these constants $C^J_M$ among all non-empty subsets $J$ of $\{1, \dots, m\}$.  Then if $\psi \in \Psi_{C_M}$, the conclusion of the previous paragraph holds.	
\end{proof}

 \begin{dfn}  \label{dfn:symmetrybreak}
A multislope for $M$ is {\em generic} if it belongs to the hyperbolic Dehn surgery space of $M$.
A multislope $\psi$  is {\em fully generic} if $\psi \in \Psi_{C_M}$. 
We further say a generic multislope for $M$ is {\em symmetry-breaking} if $g(\psi) \neq \psi$ for all non-trivial isometries $g \in Isom(M)$.
\end{dfn}

Together, Theorem~\ref{thm:fillinglengths} and Definition~\ref{dfn:symmetrybreak} immediately give the following.

\begin{cor}\label{cor:asymmetric}
If $\psi$ is a fully generic, symmetry-breaking multislope for the cusped hyperbolic $3$--manifold $M$, then the filling $M(\psi)$ is an asymmetric hyperbolic manifold. \qed
\end{cor}

\subsection{Strong inversions and unbreakable symmetries}

A manifold  $M$ is {\em strongly invertible} at a collection of boundary tori $\mathcal{T}$ if there is an  orientation-preserving order $2$ self-homeomorphism (i.e., involution) of $M$ whose  fixed locus is non-empty and meets each component of $\mathcal{T}$.   A knot or link $L$ in the manifold $\widehat{M}$ is {\em strongly invertible}  if its exterior $M=\widehat{M}-\nbhd(L)$ is strongly invertible at the collection $\mathcal{T} =\bdry\nbhd(L)$.

An element of the symmetry group of $M$ is called a {\em strong inversion (at $\mathcal{T}$)} if it is represented by an orientation-preserving involution of the kind above.  Any strong inversion of $M$  at a torus component $T \subset \bdry M$ restricts to $T$ as the hyperelliptic involution and consequently extends to an orientation-preserving involution of any Dehn filling of $M$ along $T$. 
Thus the symmetry group of a manifold that contains a strong inversion at a collection of boundary tori cannot be reduced to the trivial group via any Dehn filling of those tori.

More generally, for a manifold $M$ and a collection of boundary tori $\mathcal{T}$, any symmetry of $M$ that fixes each $T \in \mathcal{T}$ set-wise and fixes the slopes on each $T \in \mathcal{T}$  will extend across any Dehn filling of $M$ along $\mathcal{T}$, so there can be no symmetry-breaking multislopes in $\mathcal{T}$.

\section{Some families of jointly primitive knots}\label{sec:families}
In this section we introduce a collection of families of framed knots in either lens spaces or connected sums of lens spaces through surgery descriptions on a common link.  These knots are either LJP or MJP and correspondingly have surgeries to either lens spaces or connected sums of lens spaces.

\subsection{Notational conventions} 
First, let $\hatQ = \Q \cup \{\tfrac10\}$ and $\hatZ = \Z \cup \{\tfrac10\}$ denote the sets of the rational numbers and the integers, respectively, each extended by $\tfrac10=\infty$.

We employ the convention that the lens space $L(p,q)$ is obtained from $S^3$ via  $(-p\slash q)$--surgery  on the unknot or by surgery on the linear chain link of $n$ components with integer framings $a_1, \dots, a_n$, where
\[-p/q = [a_1, \dots, a_n]=a_1-1/(a_2-1/(\dots -1/a_n))\] 
as in \cite{Saveliev}. 
We also let  $L[a_1, \dots, a_n]$ denote the lens space $L(p,q)$.

Extending this notation for more general surgeries on linear chain links, we allow $a_i=\infty$ in order to describe connected sums of lens spaces.  For example, $L[2,3,\infty,4,5,6] = L[2,3]\#L[4,5,6]$.   We also allow the initial and terminal coefficients to be rational numbers so that
\[L[[s_1, \dots, s_{n'}],a_1, \dots, a_n,[r_1, \dots, r_{n''}]] = L[s_{n'},  \dots, s_1, a_1, \dots, a_n, r_1, \dots, r_{n''}]\]
 for integers $a_1, \dots, a_n, s_1, \dots, s_{n'}, r_1, \dots, r_{n''}$;  
both versions describe the same lens space by Dehn surgeries on appropriate chain links with these coefficients.   This relationship is due to the surgery calculus ``slam dunk'' move; see, for example,  \cite[Section 5.3, Figure 5.30]{gompfStipsicz1994}.

Note that these extensions of notation do not generally hold for continued fraction expansions of rational numbers.  For example,  
$\frac{18}{11}=[2,3,4] \neq [[3,2],4]=\frac{9}{4}$.

\subsection{A surgery diagram}\label{sec:surgerydiagram}
Consider the surgery diagram of Figure~\ref{fig:surgerylink} and fix $b,m,r,s \in \hatQ$. Choosing $k\in \Z$ determines a  $1$--parameter family of framed knots $K_k(m,r,s,b)$ in the closed $3$--manifold $Y(m,r,s,b)$.  Figure~\ref{fig:Ymrsb} shows how a presentation of $Y(m,r,s,b)$ as surgery on a twisted $4$--chain link may be obtained from the surgery diagram in Figure~\ref{fig:surgerylink} when the trivial surgery is performed on $K$. Note that the choice of $k$ affects the framing on $K$, but not the ambient $3$--manifold. 

\begin{figure}[h!]
	\includegraphics[width=\textwidth]{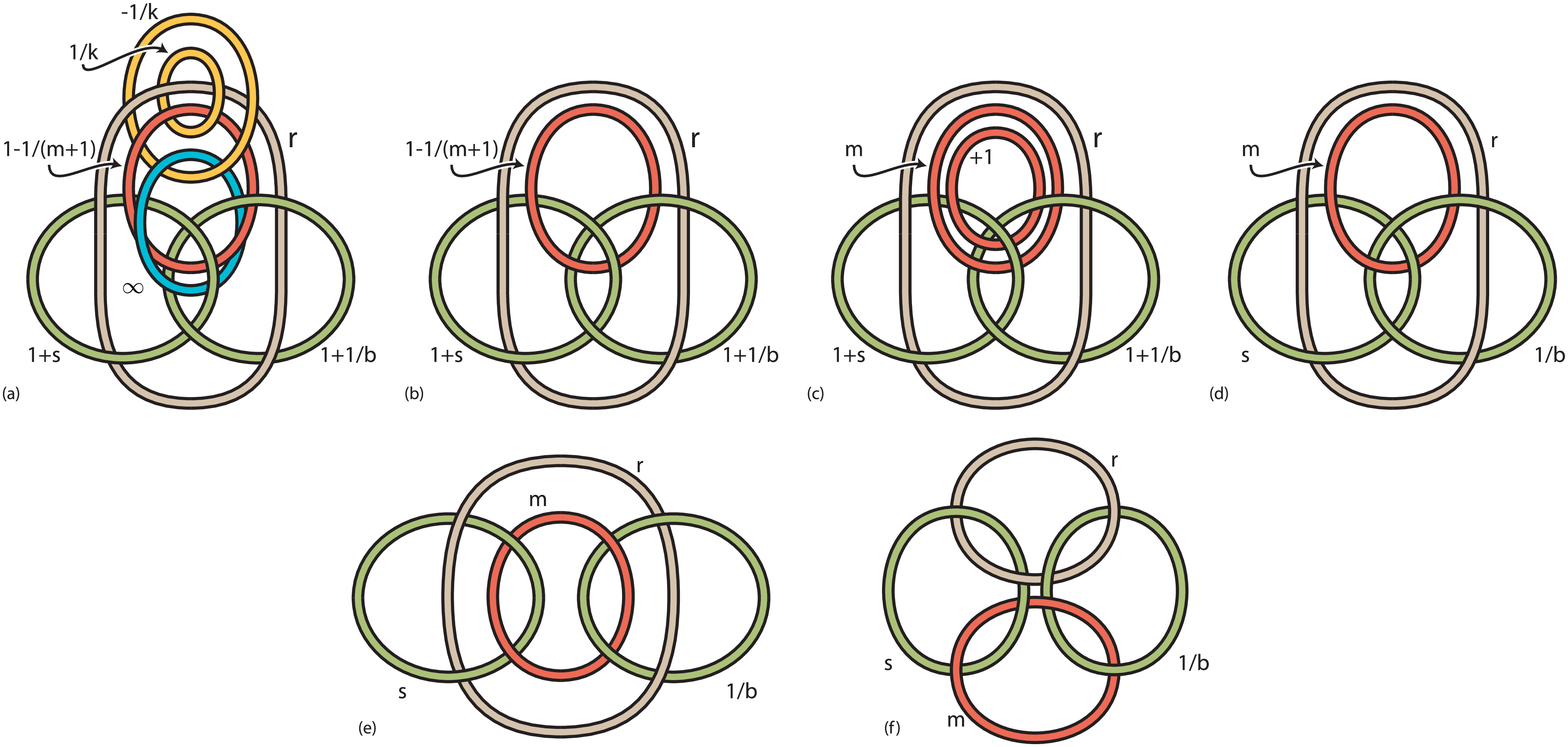}
	\caption{A simple surgery description of the ambient manifold $Y(m,r,s,b)$ is obtained from Figure~\ref{fig:surgerylink}.
	}
	\label{fig:Ymrsb}
\end{figure}
	 
Let $K_k(m,r,s,b)$ inherit  the $+1$--framing of $K\subset S^3$, and define $Y^*_k(m,r,s,b)$ to be the result of this framed surgery on $K_k(m,r,s,b)$.   It follows that 
\[Y(m,r,s,b) = S^3_{\calL}(1-\tfrac{1}{m+1},r,1+s, 1+\tfrac1b,\tfrac1k, -\tfrac1k)\] 
and $K_k(m,r,s,b)$ is the image of $K$ under this surgery, while $Y^*_k(m,r,s,b)$ is the result of a further $+1$--surgery on $K$ with surgery dual $K^*_k(m,r,s,b)$.

Let $C'_0$ and $C'_1$ be the cores of the surgeries on $C_0$ and $C_1$, respectively. Let $M(m,r) \subset Y(m,r,s,b)$ and  $M^*_k(m,r) \subset Y^*_k(m,r,s,b)$ be the exteriors of $C'_0 \cup C'_1$, i.e., the manifolds  with two torus boundary components  obtained by leaving $C_0$ and $C_1$ unfilled.   Let $K_k(m,r)$ be the corresponding framed knot in $M(m,r)$, and let $K^*_k(m,r)$ be the corresponding surgery dual framed knot in $M^*_k(m,r)$.  As shown in Figure~\ref{fig:masterlinksurface}, there is 
a twice-punctured torus $\Sigma$ bounded by $C_0 \cup C_1$ which contains $K$ as a $+1$--framed curve.   Hence we may regard $\Sigma$ as properly embedded in $M(m,r)$ and containing the framed knots $K_k(m,r)$.
We now observe that for certain choices of $m,r \in\hatQ$,
$\Sigma$  confers a jointly primitive presentation upon the knots $K_k(m,r)$ for $k\in\Z$.

\begin{lemma}\label{lem:jpcondition}
If $m=-1$ and $r \in \hatQ$, then $K_k(-1,r)$ is jointly primitive in $M(-1,r)$.   If  $m,r\in\Z$, then $K_k(m,r)$ is fibered jointly primitive in $M(m,r)$.
\end{lemma}

\begin{proof}
The proof of the lemma is an application of the Montesinos trick relating Dehn surgery on a link to rational tangle replacement in its double branched cover.  We begin by pushing the definition of jointly primitive through this correspondence before tackling the specific example.   

Suppose that the link $L\subset M$ admits an involution $\gamma$ whose fixed set intersects each component of $L$ in two points.  In the quotient, the link components become arcs with endpoints on the branch locus.  A surface in the quotient containing one of these arcs will lift to a surface $\Sigma\subset M$ containing the corresponding component $L_i$.  When attaching $2$--handles to the images of $L_i$ in $M-\nbhd(\Sigma)$ yields a handlebody, this surface confers a jointly primitive presentation, and we may evaluate this via the corresponding operations in the quotient.

Recall that a rational $n$--string tangle is any $1$--manifold properly embedded in the  $3$--ball that is homeomorphic to $n$ vertical arcs in the product $D^2\times I$.  The double branched cover of a rational $n$--string tangle is a genus $n-1$ handlebody.  Note, too, that in the double branched cover, attaching a cap to a tangle along an arc in the boundary corresponds to attaching a $2$--handle along the closed curve that is the lift of the arc.  Hence we may say a collection of arcs in the boundary of a rational $n$--string tangle is \emph{jointly primitive} if attaching caps along any subset of them produces a rational tangle.

\begin{figure}[h!]
\includegraphics[width=\textwidth]{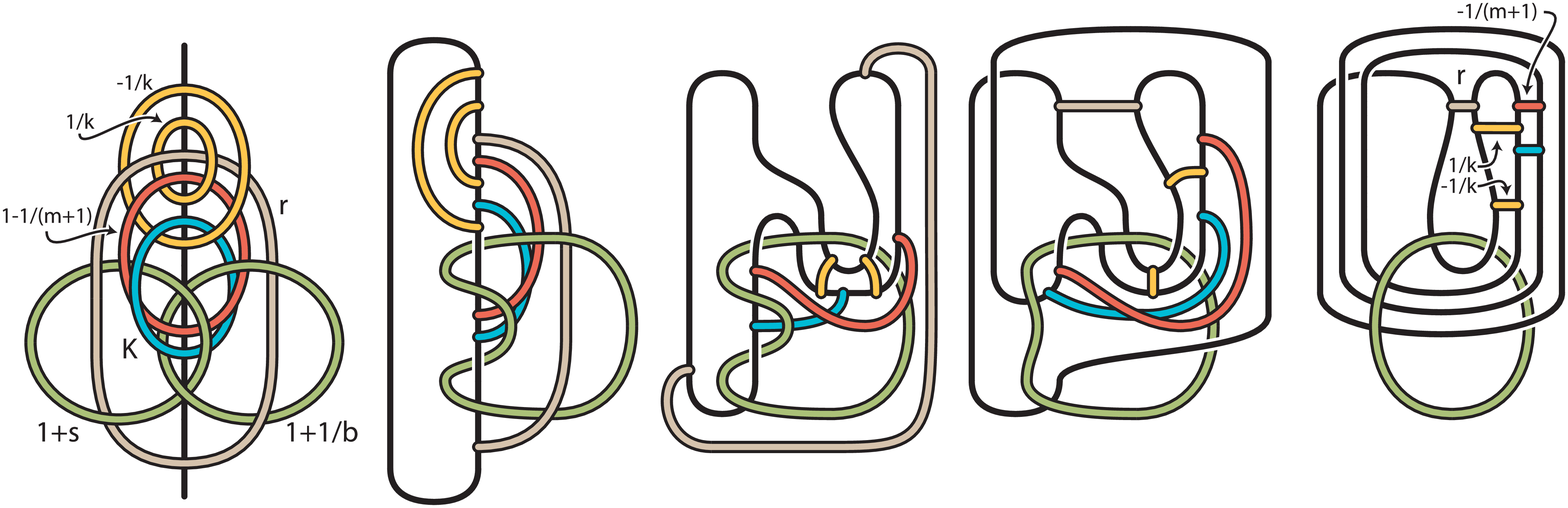}
\caption{After taking the quotient of  $K \cup \calL$ of Figure~\ref{fig:surgerylink} by an involution exchanging $C_0$ and $C_1$, the figure evolves by isotopies  and rational tangle replacement. Note that in the final isotopy, the framings of the blue and red arcs change by $1$.  
}
\label{fig:augmentedinvolution}
\end{figure}

With this definition in hand, we turn to the link $K \cup \calL$ of Figure~\ref{fig:surgerylink}, which is shown again in Figure~\ref{fig:augmentedinvolution} together with an axis of involution.  Moving left to right, the second picture  shows its quotient by the involution.   The image of the fixed set becomes the branch locus, and the quotient identifies the components $C_0$ and $C_1$ into a single unknot $C$; the other four components of $K \cup \calL$ descend to arcs in the quotient. Continuing to the right, the subsequent pictures describe a sequence of isotopies. Note that the last isotopy changes the framings on the arcs of $K$ and $J$ by $1$.  Figure~\ref{fig:augmentedinvolution2} depicts straightening out the last image of Figure~\ref{fig:augmentedinvolution} and then performing a rational tangle replacement on the image of $R$. This  transforms the branch locus $\calB$  into a $4$--string braid $\calB^*$  with braid axis $C$.   The rational tangle replacement introduced a new arc $R^*$ (dual to $R$) that is level with respect to $C$, as are the arcs coming from $K$, $L_+$, $L_-$, and $J$.  We therefore use standard parametrization of slopes for rational tangle replacements on these level arcs.   

The exterior of $C$ is a solid torus containing the link $\calB_k(m,r)$, where the parameters $m$ and $r$ designate the appropriate rational tangle replacements on the arcs of $L_+$, $L_-$,  $J$, and $R^*$.  The double cover of the solid torus branched over $\calB_k(m,r)$ is the manifold $M(m,r)$ which contains the knot $K_k(m,r)$.    Observe that $C$ bounds a disk $D$ which intersects $\calB_k(m,r)$ four times; this disk lifts to the surface $\Sigma$ containing $K_k(m,r)$.

\begin{figure}
\includegraphics[height=2.5in]{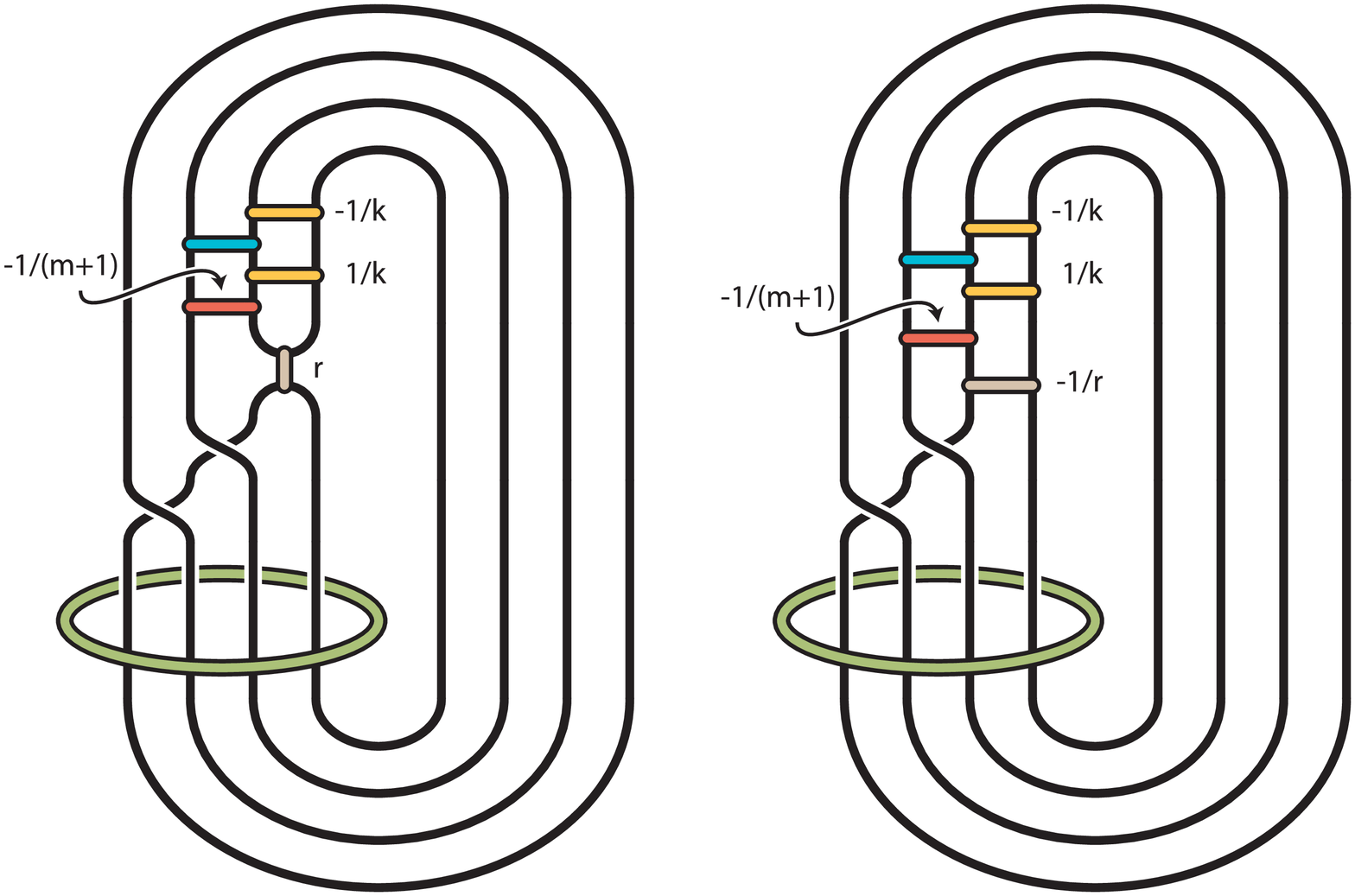}
\caption{(Left) The final image of Figure~\ref{fig:augmentedinvolution} is straightened out so that the branch locus is nearly braided. (Right) A rational tangle replacement exchanges the vertical arc with slope $r$ with a horizontal arc with slope $-\tfrac1r$.  This results in a branch locus that is presented as a closed $4$--string braid about the green braid axis $C$.  }
\label{fig:augmentedinvolution2}
\end{figure}

By construction, $H$ is the double branched cover of the ball $S^3-\nbhd(C \cup D) \cong D \times I$ branched over $\bar{\calB}_k(m,r) = \calB_k(m,r) \cap D\times I$ and the arc $\kappa$ leaves two impressions in $\bdry (D\times I)$.    This is shown for $m=-1$ and $r\in\hatQ$ in Figure~\ref{fig:jointlyprimitiveexample}(a).    We examine this case explicitly and leave the even simpler case $m,r \in \Z$ to the reader.  

\begin{figure}[h!]
\includegraphics[width=\textwidth]{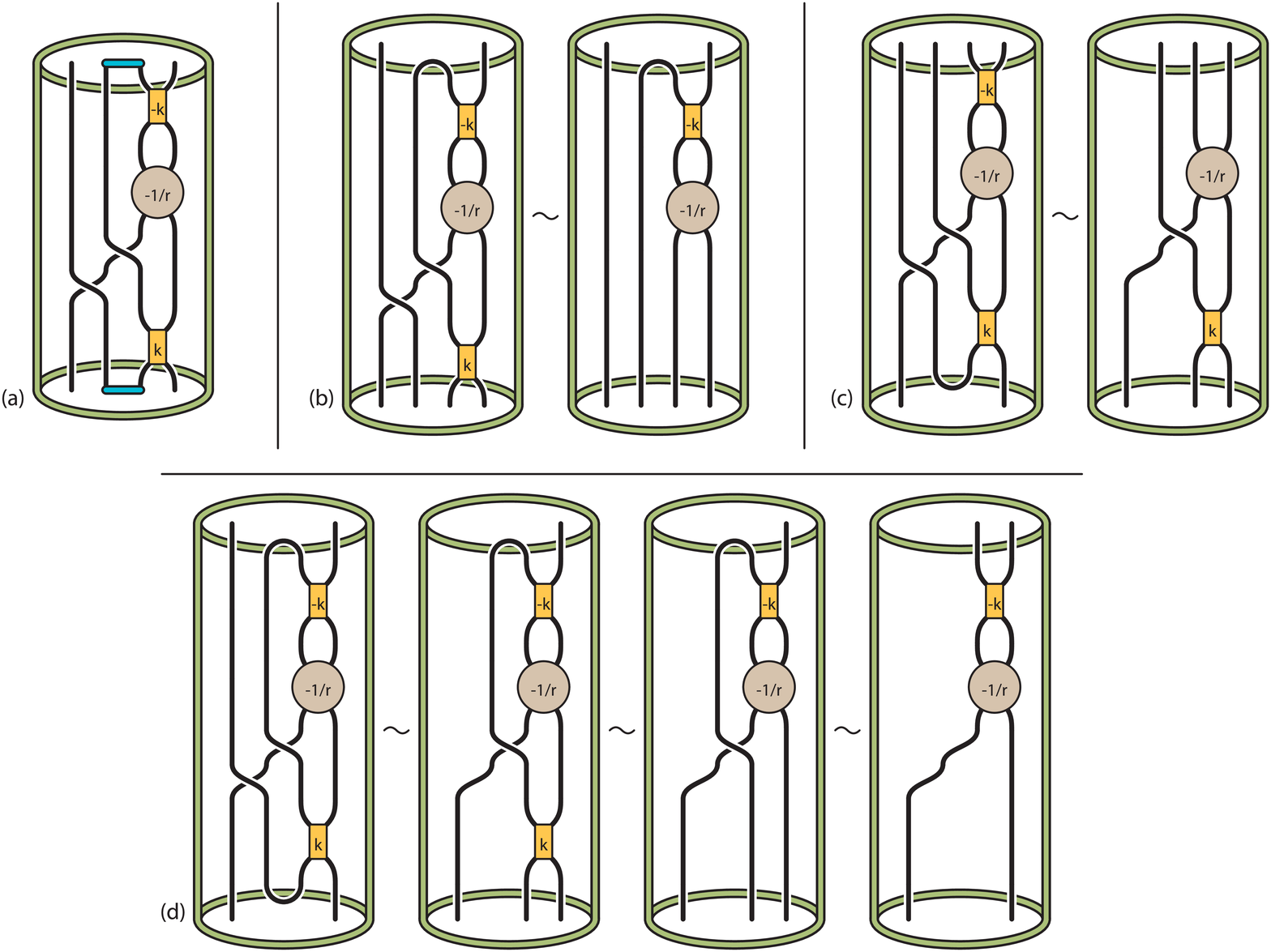}
\caption{(a) The rational tangle $\bar{\calB}_k(-1,r)$ and the rational tangles obtained by attaching caps along (b) one, (c) the other, or (d) both impressions $\kappa^+$ and $\kappa^-$ of $\kappa$.}
\label{fig:jointlyprimitiveexample}
\end{figure}
 
Since $\bar{\calB}_k(m,r)$ is a rational $4$--string tangle,  $H$ is indeed a genus $3$ handlebody.    In Figure~\ref{fig:jointlyprimitiveexample}, the label $-\tfrac1r$ on a ball indicates that it should be filled with a rational tangle of the given slope.  The rectangles indicate vertical compositions of half twists  whose number and handedness are dictated by the label.

Figures~\ref{fig:jointlyprimitiveexample}(b) and (c) each show a cap attached to $\bar{\calB}_k(m,r)$ along one impression of $K$, followed by an isotopy of the tangle into a form that is more clearly a rational $3$--string tangle. Figure~\ref{fig:jointlyprimitiveexample}(d) shows caps attached to $\bar{\calB}_k(m,r)$ along both impressions of $K$ followed by a sequence of isotopies until it is clear that the result is a rational $2$--string tangle. It follows from the general discussion at the beginning of the proof that $\Sigma$ gives a jointly primitive presentation of $K_k(m,r)$ in $M(m,r)$.
\end{proof}

\begin{lemma}\label{lem:cables}
For $m,r \in \Z$, the manifold $M_k^*(m,r)$ is the cable space $A([1,k+1,r+1,m+1,-k])$.   
	
For $m=-1$ and $r \in \hatQ$, the manifold $M_k^*(-1,r)$ is the cable space $A([1,k+1,r-k+1])$.	
\end{lemma}

\begin{proof}
Figure~\ref{fig:cablespsurg}(a) shows a tangle surgery description of a link in the solid torus exterior of the green unknot $C$ whose double branched cover is $M_k^*(m,r)$.  Figures~\ref{fig:cablespsurg}(b) and (c) manipulate this expression.  
\begin{figure}[h!]
	\includegraphics[width=\textwidth]{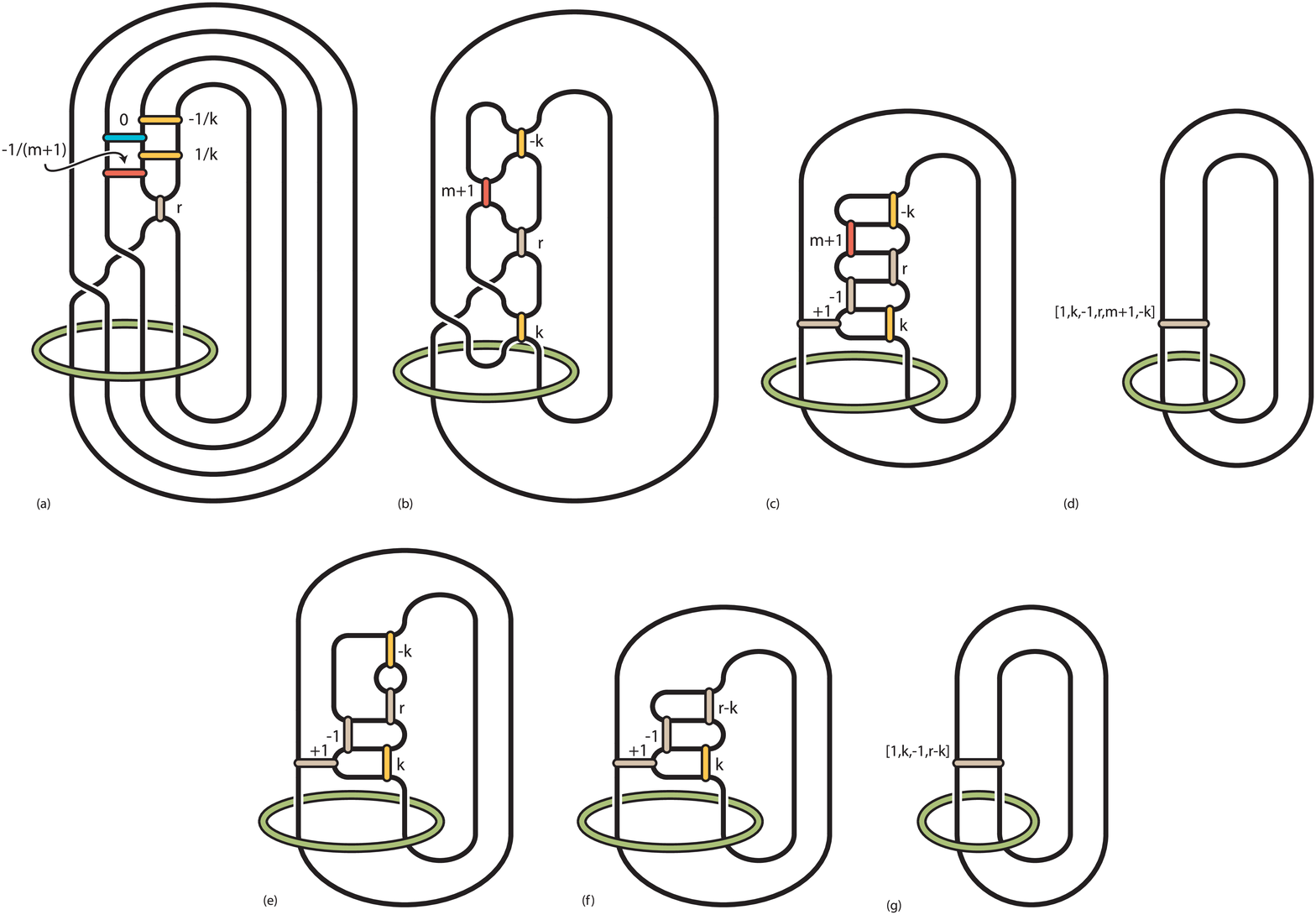}
	\caption{Figure~\ref{fig:augmentedinvolution2}(Left) is show with the slope $0$ on the arc corresponding to $K$ in (a).  The $0$ tangle replacement is performed along with the re-expression of the $-1/n$ arcs as $n$ arcs is given in (b) after an isotopy.  In (c), the two crossings are then presented as arcs with slopes $+1$ and $-1$.  Assuming that $m,r \in \Z$, in (d) we coalesce these tangle replacements into a single replacement.  Alternatively, (e) is obtained from (c) by setting $m=-1$ and performing the ensuing $0$ tangle replacement.  Since $k\in \Z$, we combine the arcs with slopes $-k$ and $r$ in (f). The arcs are further combined into a single replacement in (g).}
	\label{fig:cablespsurg}
\end{figure}

When $m,r\in \mathbb{Z}$, this becomes single arc with slope $[1,k,-1,r,m+1,-k]$ as shown in Figure~\ref{fig:cablespsurg}(d).  Using the calculus of continued fractions, this slope may be simplified to $[1,k+1,r+1,m+1,-k]$.
As the double cover of the solid torus $S^1 \times D^2$ branched over two parallel copies of the core curve is $S^1 \times (S^1 \times [-1,1])$, the arc shown in Figure~\ref{fig:cablespsurg}(d) lifts to a circle $\theta \times (S^1 \times 0)$ for some point $\theta\in S^1$.
Hence   $M_k^*(m,r) =A([1,k+1,r+1,m+1,-k])$.  
	
When $m=-1$ and $r\in\hatQ$,  Figures~\ref{fig:cablespsurg}(e),(f),(g) reduce Figure~\ref{fig:cablespsurg}(c) to  Figure~\ref{fig:cablespsurg}(d), except that  the arc now has slope $[1,k,-1,r-k]$.  This slope simplifies to $[1,k+1,r-k+1]$.  Hence $M_k^*(-1,r) = A([1,k+1,r-k+1])$.
\end{proof}

\begin{lemma}\label{lem:ljpandmjp}
Assume either $m=-1$ and $r \in \hatQ$  or $m,r \in \Z$.  If $b \in \Z$, then $K_k(m,r,s,b)$ is LJP and $Y^*_k(m,r,s,b)$ is a lens space.  If $b=\infty$, then $K_k(m,r,s,\infty)$ is MJP and $Y^*_k(m,r,s,\infty)$ is a connected sum of lens spaces.
\end{lemma}
\begin{proof}

By assumption, Lemma~\ref{lem:jpcondition} implies that $K_k(m,r)$ has a jointly primitive presentation on $\Sigma$ in $M(m,r)$.  Hence for any filling of $C_0'$ and $C_1'$, $K_k(m,r,s,b)$ has a jointly primitive presentation in $M(m,r,s,b)$.

Assume $b \in \Z$.
Since $\bdry\Sigma$ meets $C_0$ and $C_1$ with slope $+1$, any curve of slope $1+\tfrac1b$ on $C_1$ has distance $1$ from $\bdry \Sigma$.  Thus when $C'_1$ is surgery dual to $(1+\tfrac1b)$--surgery on $C_1$, $\bdry \Sigma$ will be a longitude of  $\nbhd(C'_1)$.  Hence $K_k(m,r,s,b)$ is longitudinally jointly primitive and $Y^*_k(m,r,s,b)$ is a lens space.

Similarly, when $C'_1$ is surgery dual to $1$--surgery on $C_1$, $\bdry \Sigma$ will be a meridian of $\nbhd(C'_1)$.  Note that if $b=\infty$, then $1+\tfrac1b = 1$.  Hence $K_k(m,r,s,\infty)$ is meridionally  jointly primitive and $Y^*_k(m,r,s,\infty)$ is a connected sum of lens spaces. 
\end{proof}

\subsection{Hyperbolicity and symmetries of $K_k(m,r,s,b)$}\label{sec:hypandsym}

In this section, 
we show that certain families of manifolds  are  hyperbolic and that some of them admit strong inversions, while others are (generically) asymmetric.  Our first lemma establishes a criterion for the existence of a strong inversion; however, showing that a manifold is asymmetric generally involves some computation. In service of this computation, we first use SnapPy to determine the hyperbolicity and symmetry group of the multi-cusped manifolds that each family limits to. 
  The relevant symmetry group computation follows from an analysis of the symmetries of the canonical cell decomposition of the multi-cusped manifold.  For the manifolds of interest, the canonical cell decomposition is often an ideal triangulation; in such cases we refer to it as the {\em canonical triangulation}. In the cases where the canonical cell decomposition has non-tetrahedral cells, SnapPy builds a triangulation with finite vertices  via a natural subdivision of the canonical cell decomposition. This triangulation is canonically defined and its combinatorial symmetries are in fact the full symmetry group of the manifold. We will call this triangulation the {\em canonical triangulation with finite vertices}. Moreover, SnapPy, when run as a module inside of Sage, has the functionality to rigorously compute either the canonical cell decomposition or the canonical triangulation with finite vertices for a given cusped hyperbolic 3-manifold $M$.

In the examples listed here, the isomorphism signature also includes framing data (which agrees with the figures and conventions of this paper), so the interested reader can observe our computations in SnapPy (as a Sage module) if they so desire. 
 
Thereafter we appeal to the discussion of symmetries and symmetry breaking from \S\ref{sect:symm_and_Dehn_surgery} to show that when the free parameters in each family are chosen to have suitably large magnitude, the corresponding multislope used to fill the limiting manifold is fully generic and symmetry-breaking.

\begin{figure}
	\includegraphics[width=\textwidth]{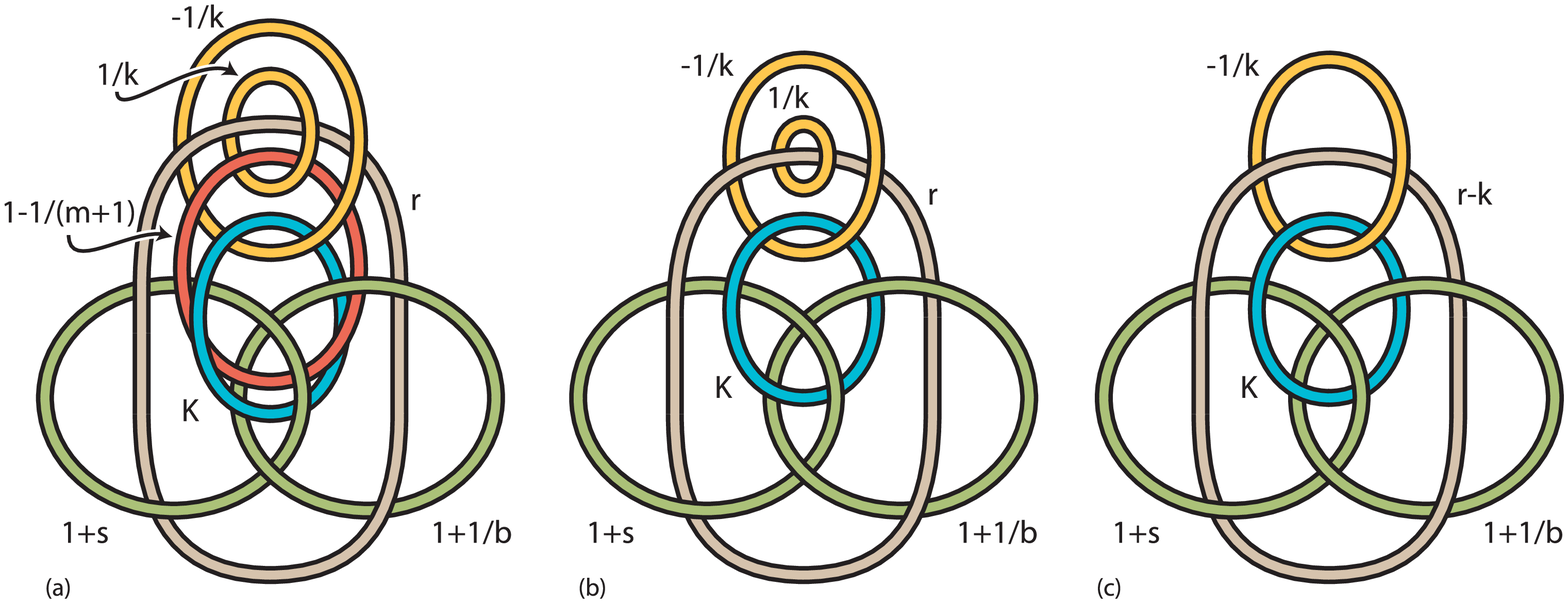}
	\caption{
	Setting $m=-1$ in (a) makes the component $J$ of $\calL$ have the surgery coefficient $\infty$.  Performing this filling causes $L_+$ to be a meridian of the component $R$, so that $K \cup \calL$ becomes the toroidal link in (b).  However since $L_+$ has surgery coefficient $1/k$, we may use Rolfsen twists to ``absorb'' $L_+$ into $R$ so that the resulting link in (c) is hyperbolic (in fact, it is L12n2253).  {Continuing to use the labeling of the components from the original link $K \cup \calL$, this is the $5$--component link $K \cup R \cup C_0 \cup C_1 \cup L_-$.}  With its surgery coefficients, it describes the framed knot  $K_k(-1,r,s,b)$. } 
	\label{fig:m=-1}
\end{figure}

\begin{lemma}\label{lem:strinv}
	$K_k(m,r,s,b)$ is strongly invertible if  $s\in\{0,-1,-2,  \infty\}$ or $b\in\{\infty,-1,-1/2, 0\}$. 
\end{lemma}

\begin{proof}
	Since the rotation of $\calL \cup K$ along the vertical axis in Figure~\ref{fig:surgerylink}  exchanges $C_0$ and $C_1$ and is strong on the remaining components, we observe that $K_k(m,r,s,b) = K_k(m,r,1/b,1/s)$.  Furthermore $K_k(m,r,s,b)$ is strongly invertible for $s \in \{0,-1,-2, \infty\}$ if and only if it is strongly invertible  for the corresponding $b=1/s\in\{\infty,-1,-1/2,  0\}$.  Therefore we only check that performing $1+s$ surgery on $C_0$ produces a strongly invertible link for the four cases $s=\infty$, $s=0$, $s=-1$,  and $s=-2$.
\begin{figure}
\includegraphics[width=\textwidth]{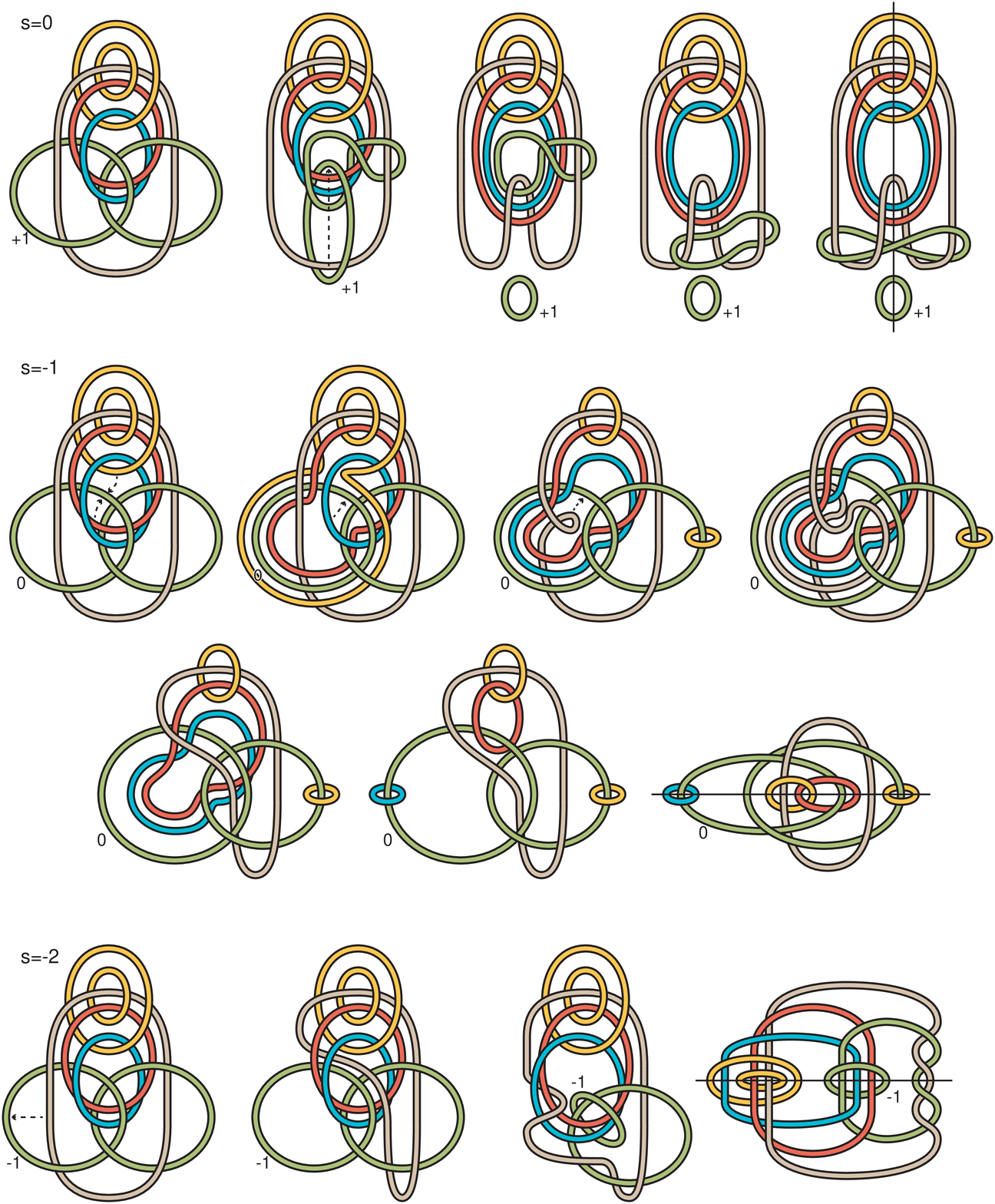}
\caption{In each case $s \in \{0,-1,-2\}$, the link $\calL \cup K$ of Figure~\ref{fig:surgerylink} is first shown with surgery coefficient $1+s$ on $C_0$.  Then sequences of handleslides of other components over $C_0$ and isotopies transform the link into one with an involution that is clearly strongly invertible on all components, except in the case $s=0$ where it is instead freely invertible on the component corresponding to $C_1$. }  
\label{fig:stronglyinvertiblesurgeries}
\end{figure}

	When $s=\infty$, we may omit component $C_0$ from Figure~\ref{fig:surgerylink}. Letting $C_1$ ``hang freely'' upon the other components, one sees that this link can be positioned to be strongly invertible along a vertical axis. For the three cases $s \in \{0,-1,-2\}$, Figure~\ref{fig:stronglyinvertiblesurgeries} shows the results of handleslides over $C_0$ and isotopies in each case.  The resulting links are strongly invertible in the cases $s=-1$ and $s=-2$, while in the case $s=0$, the inversion is free on $C_1$ and strong on the other components.  
\end{proof}

We now employ  rigorous computation to find asymmetric manifolds.

\begin{lemma}\label{lem:bulksymmetry}
For $n,b,k\in \Z$ of suitably large magnitude and $(r,s) \in \{(-1, -4+\tfrac1n), (-2,-3+\tfrac1n)\}$, the manifolds $X_k(-1,r,s,b)$ are hyperbolic and  asymmetric.
\end{lemma}

\begin{proof}
Let $M$ be the complement of the five-component link $K \cup R \cup C_0 \cup C_1 \cup L_-$ shown in Figure~\ref{fig:m=-1}(c) and consider the multislope $\phi_{k,r,s,b} = (*, r-k,1+s,1+1/b,-1/k)$.  According to Figure~\ref{fig:m=-1}, this filling yields the manifold $X_k(-1,r,s,b)$.

SnapPy run as a Sage module, determines that $M$ is a hyperbolic manifold with 
canonical triangulation given by

\begin{verbatim}
sage: import snappy
sage:  M = snappy.Manifold('qvvLLLAAPQQkfjlhmkjnlmknpopopaaaagocagggaaaaco_abdecBbb
....: CaBbCBbBabaababBa')
sage: M.verify_hyperbolicity()[0]
sage: MC = M.canonical_retriangulation(True)
sage: len(MC.isomorphisms_to(MC)) == M.symmetry_group().order()
sage: M.symmetry_group().isometries()
\end{verbatim}

Here, we drew the link in SnapPy
(see figure10c.lnk in the ancillary files)
 and then imported it into sage  via its triangulation signature. From this computation, we see the symmetry group of the original triangulation is the full symmetry group and there is a unique non-trivial symmetry $\tau$ which fixes the cusp of the component $K$.  This symmetry is an involution that exchanges two cusps (indexed by `2' and `3' respectively in the associated triangulation $M$) and is a  strong involution on the remaining cusps. It can be observed in Figure~\ref{fig:m=-1}(c) as rotation about a vertical axis.  One then sees, using the standard meridian-longitude parameterizations of surgery slopes for the exchanged cusps (formerly $C_0$ and $C_1$), that the involution $\tau$ exchanges the $p/q$ slope of one with the $p/q$ slope of the other.

By choosing $n,b,k \in \Z$ of suitably large magnitude, we may ensure that $r-k$ and $s$ also have suitably large magnitude to form a fully generic multislope $\phi_{k,r,s,b} = (*, r-k,1+s,1+1/b,-1/k)$ in $\Psi_{C_M}$.  Theorem \ref{thm:fillinglengths} informs us that for any multislope  $\psi \in \Psi_{C_M}$, any element of Isom($M(\psi)$) restricts to an isometry of $M$ on the complement of the cores of the surgery solid tori.   Yet because 
$1+(-4+\tfrac1n)\ne 1+\tfrac1b$ and $1+(-3+\tfrac1n)\ne 1+\tfrac1b$ for any $b,n \in \Z$, the involution $\tau \in \Isom(M)$ cannot be a restriction of a symmetry in $\Isom(M(\phi_{k,r,s,b}))$.  So $\phi_{k,r,s,b}$ is symmetry-breaking.  Hence $\Isom(M(\phi_{k,r,s,b}))$ must be trivial.  Thus $X_k(-1,r,s,b)$ is asymmetric.
\end{proof}

\begin{lemma}\label{lem:bulkreduciblesymmetry}
For $b,k\in \Z$ of suitably large magnitude and $(r,s) \in \{(-1, -4), (-2,-3)\}$, the manifolds $X_k(-1,r,s,b)$ are hyperbolic and asymmetric. 
\end{lemma}

\begin{proof}
Continue with the notation and idea of the previous proof.
Let $N_{-4}$ be the manifold $M(*, *,-3,*,*)$ and consider the multislope $\phi_{k,b} = (*, -1-k,1+1/b,-1/k)$.  Then $N_{-4}(\phi_{k,b}) = M(\phi_{k,-1,-4,b}) = X_k(-1,-1,-4,b)$.

SnapPy determines that $N_{-4}$ is a hyperbolic manifold with canonical triangulation given by \begin{verbatim}
'rvLLLzLzQMQQcdgiijjklmknppooqqqoauaccvvtjlagggvvv_dacbaBbbaBbBaBbaBbBa'.\end{verbatim}
SnapPy also determines that this canonical triangulation (and hence, $\Isom(N_{-4})$) has just one non-trivial symmetry. However, that symmetry exchanges the cusp corresponding to $K$ with the cusp corresponding to $R$. As before,  choosing $b,k\in \Z$ of suitably large magnitude  ensures that $\phi_{k,b} \in \Psi_{C_{N_{-4}}}$ is fully generic and symmetry-breaking.  Hence we conclude that $\Isom(N_{-4}(\phi_{k,b}))$ is trivial and $X_k(-1,-1,-4,b)$ is asymmetric.
	
\medskip

Now let $N_{-3}$ be the manifold $M(*, *,-2,*,*)$ and consider the multislope $\phi'_{k,b} = (*, -2-k,1+1/b,-1/k)$.  Then $N_{-3}(\phi'_{k,b}) = M(\phi_{k,-2,-3,b}) = X_k(-1,-2,-3,b)$.

\begin{verbatim}
'nvvLAPPMQkghfkjiljkljmmmmtadegrojrghqq_dcabBabbBbBcaBBbbabB'
\end{verbatim}

then show that the isomorphisms of this triangulation are in fact the whole group.

\begin{verbatim}
sage: M = snappy.Manifold('qvvLLLAAPQQkfjlhmkjnlmknpopopaaaagocagggaaaaco_abcedBbbCaBbCBaaBbBbaaBb\
a')
sage: M.dehn_fill((-2,1),2)
sage: MF = snappy.Manifold(M.filled_triangulation())
sage: CF = MF.canonical_retriangulation(True)
sage:  "CF a triangulation:", not CF.has_finite_vertices()
sage: M2 = snappy.Manifold('nvvLAPPMQkghfkjiljkljmmmmtadegrojrghqq_dcabBabbBbBcaBBbbabB')
sage: MF.verify_hyperbolicity()[0]
sage: print "isometry check:", MF.is_isometric_to(M2)
sage: print "symmetry group check:", len(CF.isomorphisms_to(CF)) == len(M2.isomorphisms_to(M2))
sage: print "CF.isomorphisms_to(CF)", CF.isomorphisms_to(CF)
sage: print "len(isom(CF):", len(CF.isomorphisms_to(CF))
\end{verbatim}

SnapPy also determines that this triangulation $M2$ (and hence $\Isom(N_{-3})$)
has a single non-trivial symmetry that preserves the cusp corresponding to $K$.  However, it is an orientation-reversing symmetry which fixes all the cusps.  Thus each cusp has a single slope preserved by this symmetry.  Regardless, $\phi'_{k,b}$ is not preserved by this symmetry so it must be symmetry-breaking. Hence choosing $b,k\in \Z$ of suitably large magnitude  to ensure that $\phi'_{k,b} \in \Psi_{C_{N_{-3}}}$ is fully generic ensures that
 $\Isom(N_{-3}(\phi'_{k,b}))$ is trivial and $X_k(-1,-2,-3,b)$ is asymmetric.   
\end{proof}

We now check the following manifolds.

\begin{lemma}\label{lem:bulk2symmetry}
For $b,k\in \Z$ of suitably large magnitude and $(m,s) \in \{(-1,-6), (-2,-4), (-3,-3)\}$, the manifolds $X_k(m,0,s,b)$ are hyperbolic
and asymmetric.
\end{lemma}

\begin{proof}

	We begin by considering the case  $m=-1$, which allows us to work with the link in Figure \ref{fig:m=-1}(c) and its complement $M$ as in the proof of Lemma~\ref{lem:bulksymmetry}. We then set  $s=-6$ to fill the cusp corresponding to $C_0$. 
	Let $N_{-6}$ be the manifold $M(*, *,-5,*,*)$ and consider the multislope $\phi''_{k,b} = (*,-k,1+1/b,-1/k)$.  (Note that $r=0$ for this lemma.) Then $N_{-6}(\phi''_{k,b}) = M(\phi_{k,0,-6,b}) = X_k(-1,0,-6,b)$.
	
	 A SnapPy computation similar to those above (see lemma4\_7check.py in the ancillary files) shows that $N_{-1,-6}$ is hyperbolic and has a single non-trivial symmetry.   This symmetry is a strong inversion on the cusps corresponding to $R$ and $C_1$ but exchanges the cusps corresponding to $K$ and $L_-$.  Hence the multislope $\phi''_{k,b}$ is not preserved by this symmetry.  Hence when $b,k \in \Z$ have suitably large magnitude to ensure $\phi_{k,b} \in \Psi_{C_{N_{-1,-6}}}$ is fully generic and symmetry-breaking, we conclude that $\Isom(N_{-1,-6}(\phi''_{k,b}))$ is trivial and $X_k(-1,0,-6,b)$ is an asymmetric hyperbolic manifold.
	
\medskip

We next consider manifolds of the form $X_k(m,0,s,b)$, where  $(m,s)$ is in $\{(-2,-4), (-3,-3)\}$. We first fill the seven-cusped manifold in Figure \ref{fig:surgerylink} along cusps $C_0$, $J$ and $R$ to produce $4$--cusped manifolds $N_{-2,-4}$ and $N_{-3,-3}$. In both cases, the filled manifolds can be verified {with a SnapPy computation as in the proof of Lemma~\ref{lem:bulksymmetry}} as hyperbolic and asymmetric.   Finally consider the remaining multislopes $(*,1+1/b, 1/k, -1/k)$ filling $N_{m,s}$ along the cusps corresponding to $C_1$, $L_+$, and $L_-$ to produce $X_k(m,0,s,b)$.  Then for $(m,s) \in \{(-2,-4), (-3,-3)\}$, when  $b,k \in \Z$ have suitably large magnitude to ensure that these slopes are in $\Psi_{C_{N_{m,s}}}$ and hence fully generic, the manifold $X_k(m,0,s,b)$ is an asymmetric hyperbolic manifold.
\end{proof}

For the next lemma, we  consider fillings that are fully generic  in the sense of \S\ref{sec:isometrygroupsfillings}. In particular, we fill along slopes with sufficiently high fillings so that the symmetry group of the filled manifold injects into the symmetry group of the unfilled manifold.

\begin{lemma}\label{lem:tableInfo}
For $k \in \Z$ of suitably large magnitude and $(m,s,b)$ as in Table~\ref{table:cabledgofk2}, the manifolds $X_k(m,0,s,b)$ are hyperbolic.  Furthermore they are generically asymmetric except where noted in the table.
\end{lemma}

\begin{proof} 
Each of these manifolds can be obtained from Figure \ref{fig:surgerylink} according to the filling parameters indicated in Table~\ref{table:cabledgofk2}. The code table4check.py run in sage verifies all of the data in the table. 
\end{proof}

\subsection{Surgeries between cable spaces}
Given a knot surgery between Seifert fibered manifolds with boundary, the various fillings of the boundary components will induce families of surgeries between closed (generalized) Seifert fibered manifolds.  For example, Berge's knots in solid tori with solid torus surgeries \cite{bergeS1xD2} naturally generate families of knots in lens spaces with lens spaces surgeries.   At its core, the ``seiferter'' theory of Deruelle-Miyazaki-Motegi \cite{DMM} similarly uses knots in solid tori with surgeries to Seifert fibered spaces over the disk with two exceptional fibers to create families of knots in $S^3$ with surgeries to small Seifert fibered spaces.  

More generally, one might hope to find knots in $T^2 \times I$ with a non-trivial surgery to $T^2 \times I$. However, no such hyperbolic knot exists (see \cite[Lemma 5.2]{BBCW2012} for example). Instead, we set out to find knots in cable spaces which admit a non-trivial surgery to another cable space (including $T^2 \times I$).  To do this, we begin with a focus on finding fibered jointly primitive knots in cable spaces.  Hence we first determine which cable spaces are twice-punctured torus bundles.

\begin{lemma}\label{lem:2PTcable}
	Up to homeomorphism, the manifolds $A(2/1)$ and $A(3/1)$ are the only cable spaces which are twice-punctured torus bundles. 
\end{lemma}

\begin{proof}
	Consider the Thurston norm $x$ on $H_2\big(A(p/q), \bdry A(p/q); \R\big) \cong \R^2$ \cite{thurstonnorm}.    View $A(p/q)$ as the exterior of the $(p,q)$--torus knot in the solid torus,  let $D$ be a $p$--punctured meridional disk, and let $A$ be the spanning annulus.  Then $[D]$ and $[A]$ form a basis for $H_2(A(p/q), \bdry A(p/q))$.  Since $x([A])=0$, the Thurston norm ball is a strip with boundary a pair of parallel lines.  
	Since $D$ is a fiber of a fibration of $A(p/q)$, these parallel lines are ``fibered faces''.  Hence any integral class that is not a multiple of $[A]$ is represented by a fiber of a fiber bundle.  Moreover, this implies that any fiber $F$ of a fibration of $A(p/q)$ must be represented as $F=mD+nA$ for integers $m,n$ where $m \neq 0$.  
	
	Since $mD+nA$ and $mD+(n+m)A$ are related by a homeomorphism of $A(p/q)$ (indeed, by a Dehn twist along $A$), we only need to consider $n \pmod{m}$.  By mirroring and reversals of orientation we may assume that $m$ and $p$ are positive integers.  Then, since $x([D])  = p-1$, we have that $x(F)= x(m[D]+n[A]) = m(p-1)$.  This can equal $2$ only if either $(m,p)=(2,2)$ or $(1,3)$.  For $(m,p)=(2,2)$, one  checks among the $n \pmod 2$ that $F$ is a twice-punctured torus only if $n=1$ and $F=2D+A$. For $(m,p)=(1,3)$, one similarly checks among the $n \pmod 3$ that, for a certain choice of orientation on $A$, $F$ is a twice-punctured torus only if $n=1$ and $F=D+A$.  Thus, up to homeomorphism, the only cable spaces that are twice-punctured torus bundles are $A(2/1)$ and $A(3/1)$.    
\end{proof}

The next result employs notation introduced in  Section~\ref{sec:surgerydiagram}.

\begin{theorem}\label{thm:bulkcable}
The manifolds $M(-1,-1)$ and $M(-1,-2)$ are the cable spaces $A(2/1)$ and $A(3/1)$, respectively.  
For each $k\in\Z$, the framed knots $K_k(-1,-1)$ and $K_k(-1,-2)$ have fibered jointly primitive presentations in these cable spaces and determine surgeries to cable spaces 
$A([1,k+1,-k])$ and $A([1,k+1,-k-1])$, respectively.
\end{theorem}

\begin{proof}
Observe that the closed $4$--braid $\calB_k(-1,-1)$ is the $(4,-1)$--torus knot in the solid torus $V$ and hence a fiber of a Seifert fibration of $V$.   Hence its double branched cover $M(-1,-1)$ is a Seifert fibered space and a twice-punctured torus bundle. Since the braid has  periodic monodromy of order $4$, the bundle does too. As the annulus from $\calB_k(-1,-1)$ to $\bdry V$ has a solid torus complement, it lifts to an annulus separating $M(-1,-1)$ into two solid tori.  Hence $M(-1,-1)$ is a cable space and by Lemma~\ref{lem:2PTcable}, we identify it as $A(2/1)$.  We note that this can alternatively be seen from a surgery description of $M(-1,-1)$ as the exterior of the two green components in Figure~\ref{fig:firstsurgery}(a).  The transformations (b) and (c) preserve this manifold.  Following Figure~\ref{fig:cablespace}(b), this gives the cable space $A(-2/1)$ which is orientation preservingly homeomorphic to $A(2/1)$.

The closed $4$--braid $\calB_k(-1,-2)$ is the $(3,-1)$--torus knot in $V$ together with the core of $V$ as the fourth strand. As such, this braid is a union of a regular fiber and an exceptional fiber in a Seifert fibration of $V$. 
 Hence its double branched cover $M(-1,-2)$ is a Seifert fibered space and a twice-punctured torus bundle. Since the braid has  periodic monodromy of order $3$, the bundle does too. 
As the annulus from $\calB_k(-1,-2)$ to $\bdry V$ has a solid torus complement, it lifts to an annulus separating $M(-1,-2)$ into two solid tori.  Hence $M(-1,-2)$ is a cable space and by Lemma~\ref{lem:2PTcable} we identify it as $A(3/1)$ up to homeomorphism.

Alternatively, we can determine the orientation of $M(-1,-2)$ from its surgery description as the exterior of the two green components in Figure~\ref{fig:firstsurgery}(f).  The transformations through (k) preserve this manifold.  First a twist along the left green component changes the linking brown component's surgery coefficient from $2/3$ to $-1/3$.  Then a handleslide of this $-1/3$ sloped component over the $0$--framed component produces the linear five-component chain link exterior with fillings $-1/3$, $0$, and $2$ on the second, fourth, and fifth components.  A slam dunk collapses the fifth into the fourth producing a meridian of the third with surgery coefficient $-1/2$.  Since the third component is unfilled, this performing this $-1/2$ surgery doesn't affect the manifold. Following Figure~\ref{fig:cablespace}(b), this gives the cable space $A(3/1)$.

\medskip

Lemma~\ref{lem:jpcondition} shows that the framed knots $K_k(-1,-1)$ and $K_k(-1,-2)$ have fibered jointly primitive presentations in the manifolds $M(-1,-1)$ and $M(-1,-2)$ respectively.  The framed surgeries on these knots produce the manifolds $M^*_k(-1,-1)$ and $M^*_k(-1,-2)$ which, by Lemma~\ref{lem:cables}, are the cable spaces $A([1,k+1,-k])$ and $A([1,k+1,-k-1])$.
\end{proof}

\begin{theorem}\label{thm:bulk} 
	For $b,n,k\in\Z$ and $(r,s) = (-1, -4+\tfrac1n)$ or $(-2,-3+\tfrac1n)$, the framed knots 
	$K_k(-1,r,s,b)$
	are LJP knots in the lens spaces $Y(-1,r,s,b)$ and give surgeries to the lens spaces $Y^*_k(-1,r,s,b)$ as detailed in the first two rows of Table~\ref{table:cablespacemain}.  Furthermore, the knots $K_k(-1,r,s,b)$ are generically hyperbolic and asymmetric.

	Allowing $n=\infty$  or $b=\infty$, these framed knots extend to LJP knots or MJP knots (according to whether $b\in\Z$ or $b=\infty$) in $Y(-1,r,s,b)$ and give surgeries to $Y^*_k(-1,r,s,b)$ as detailed in the rest of Table~\ref{table:cablespacemain}.
\end{theorem}

\begin{table}
	\small
	\caption{The knots $K_k(-1,r,s,b)$ with parameters below are obtained from the framed knots $K_k(-1,-1)$ and $K_k(-1,-2)$. They give surgeries between the manifolds $Y(-1,r,s,b)$ and $Y^*_k(-1,r,s,b)$ which are lens spaces or connected sums of lens spaces.  Where marked with \S, the knots $K_k(-1,r,s,b)$ are strongly invertible by Lemma~\ref{lem:strinv}.
	}
	\label{table:cablespacemain}
	\begin{tabular}[t]{@{}llrrrrr@{}}
		\toprule
		$Y(-1,r,s,b)$ &   $Y^*_k(-1,r,s,b)$  & $r$ & $s$ & $b$  & \\
		\midrule
		$L[2,-n,4,-b]$ &   $L[-n,-4,-b-1,k,-k]$     & $-1$  &  $-4+\frac{1}{n}$  &  $b$  \\
		$L[3,1-n,3,-b]$ &   $L[-n,-3,-b-1,k,-k-1]$     & $-2$  &  $-3+\frac{1}{n}$  &  $b$  \\
		\midrule
		$L[2]\#L[4,-b]$ &   $L[-4,-b-1,k,-k]$      & $-1$  &  $-4$  &  $b$  \\
		$L[3]\#L[3,-b]$ &   $L[-3,-b-1,k,-k-1]$      & $-2$  &  $-3$  &  $b$  \\
		\midrule
		$L[2,1-n,4]$ &   $L[-n,-4]\#L[k,-k]$      & $-1$  &  $-4+\frac{1}{n}$  &  $\infty$  & \S \\
		$L[3,1-n,3]$ &   $L[-n,-3]\#L[k,-k-1]$      & $-2$  &  $-3+\frac{1}{n}$  &  $\infty$ & \S  \\
		\midrule
		$L[2]\#L[4]$ &   $L[-4]\#L[k,-k]$     & $-1$  &  $-4$  &  $\infty$  & \S \\
		$L[3]\#L[3]$ &   $L[-3]\#L[k,-k-1]$       & $-2$  &  $-3$  &  $\infty$ & \S  \\	    
		\bottomrule
	\end{tabular}
\end{table}

\begin{proof}[Proof of Theorem~\ref{thm:bulk}]
Let us take $e \in \{0,1\}$ so that $(r,s) = (-1, -4+\tfrac1n)$ or $(-2,-3+\tfrac1n)$ may be rewritten as $(r,s) = (-1-e,-4+e+1/n)$.  Then we may consider both families of knots together as $K_k(-1, -1-e, -4+e+\tfrac1n, b)$.
Since  $m,r \in \Z$ and $b\in\hatZ$, Lemmas~\ref{lem:jpcondition} and \ref{lem:ljpandmjp} imply that $K_k(-1, -1-e, -4+e+\tfrac1n, b)$ has a fibered jointly primitive presentation that is longitudinal if $b\in\Z$ and meridional if $b=\infty$.
Kirby calculus demonstrates that $Y(-1,-1-e,-4+e+\tfrac{1}{n}, b)$ and $Y^*_k(-1,-1-e,-4+e+\tfrac{1}{n}, b)$ are the stated manifolds are given in Figures~\ref{fig:firstsurgery} and \ref{fig:secondsurgery}.  

The claims of generic hyperbolicity and asymmetry of these knots for $b,n,k \in \Z$ are established by Lemma~\ref{lem:bulksymmetry}.  When $b,k\in\Z$ but $n=\infty$, the generic hyperbolicity and asymmetry is established by Lemma~\ref{lem:bulkreduciblesymmetry}.   Finally, when $b=\infty$, the knots $K_k(-1,r,s,\infty)$ are strongly invertible by Lemma~\ref{lem:strinv}. 
\end{proof}

\begin{figure}
\includegraphics[width=\textwidth]{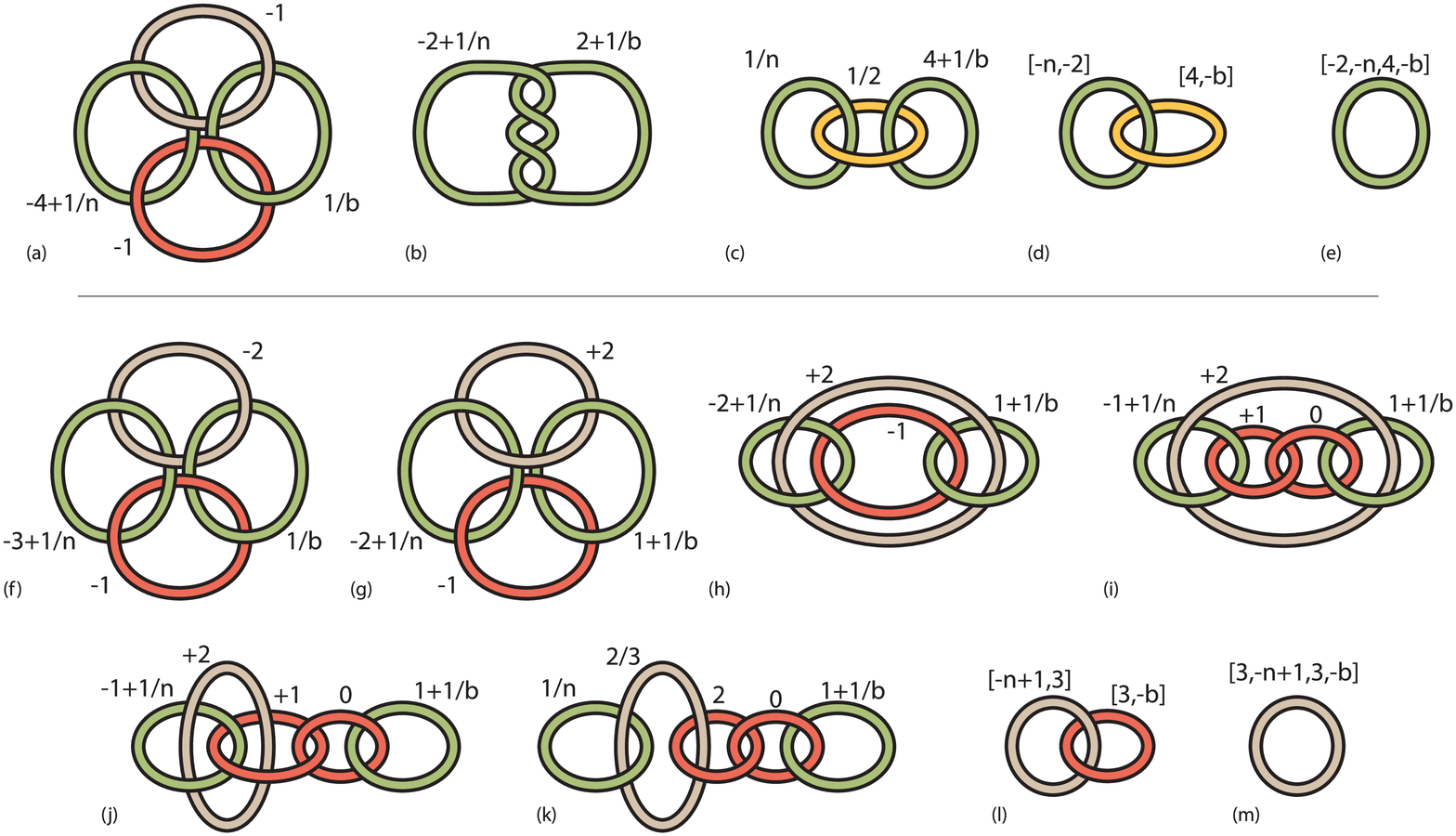}
\caption{The manifolds $Y(-1,-1,-4+1/n,b)$ and $Y(-1,-2,-3+1/n,b)$ are determined through surgery calculus, starting from the surgery diagram of $Y(m,r,s,b)$ of Figure~\ref{fig:Ymrsb}:  For $Y(-1,-1,-4+1/n,b)$, we set $m=r=-1$ and $s=-4+1/n$ in (a).  Then the two $(-1)$--framed unknots are blown down to get (b).  Rolfsen twists give (c) and a slam dunk gives (d) which we may rewrite as (e).  For $Y(-1,-2,-3+1/n,b)$, we set $m=-1$, $r=-2$, and $s=-3+1/n$ in (f). A blow up on a clasp followed by a blow down on the image of the $(-2)$--framed unknot produces (g).  After an isotopy to (h), a blow up gives (i) and then a handleslide gives (j).  A Rolfsen twist produces (k) and then twists and  slam dunks give (l) which we may rewrite as (m).
\label{fig:firstsurgery}  }
\end{figure}

\begin{figure}[h!]
\centering
\includegraphics[width=\textwidth]{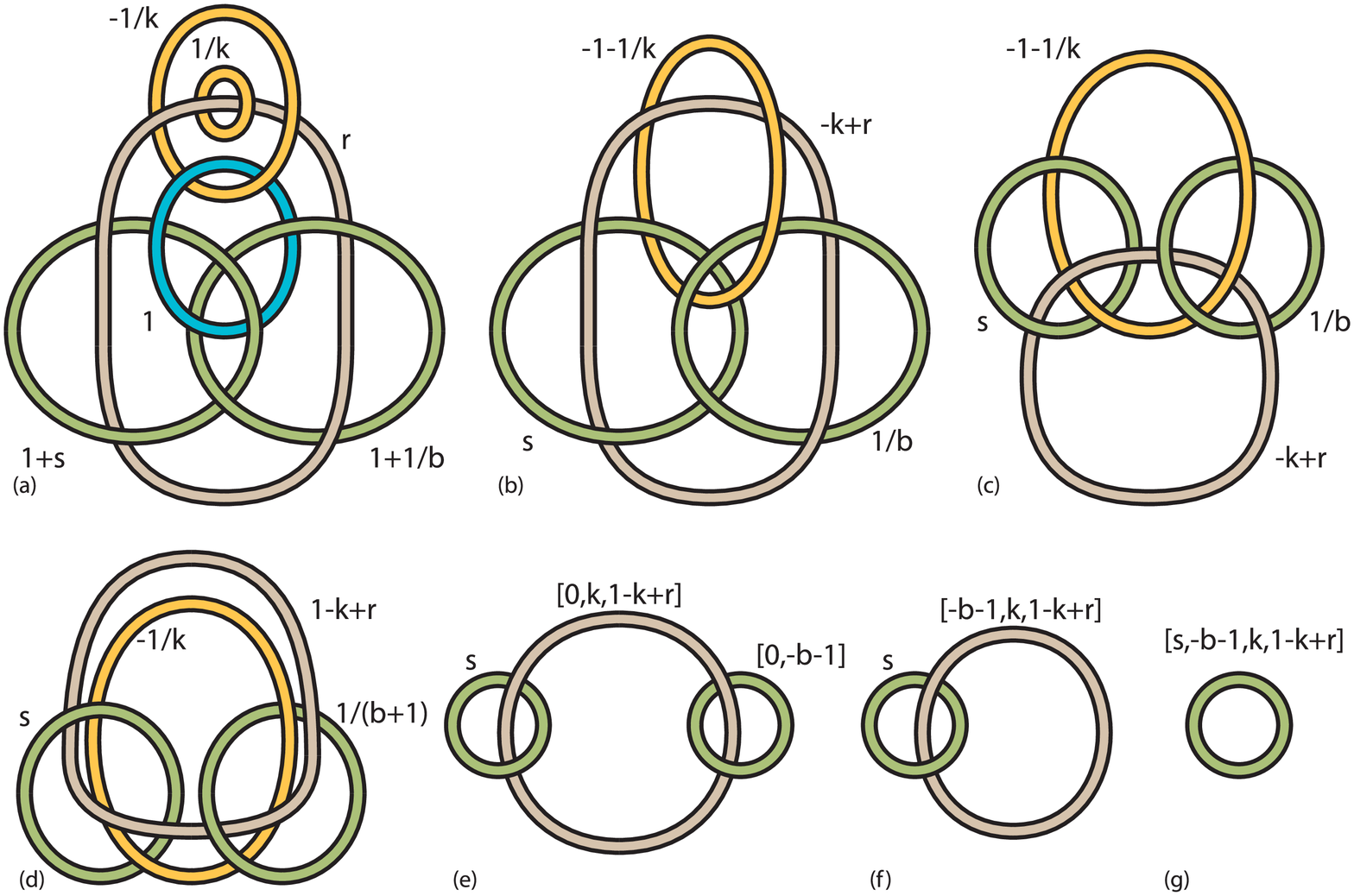}
\caption{With $m=-1$, framed surgery on $K_k(-1,r,s,b)$ produces the manifold $Y^*_k(-1,r,s,b)=L[s,-b-1,k,1-k,0,r]$: Starting from the surgery diagram of $K_k(m,r,s,b)$ of  Figure~\ref{fig:surgerylink}, setting $m=-1$ causes the trivial surgery on $J$, so it is omitted in (a).  The for the framing of our knot, the surgery coefficient $+1$ is also placed on $K$.  Next, twist $L_+$ into $R$ and blow down $K$ to get (b).  An isotopy gives (c). A Rolfsen twist of the lower right component gives (d) in which two components are parallel.  The merging of these two components gives (e) in which continued fraction notation is now used on the central component.  With $b \in \Z$, twist the righthand component into the central to get (f). Our notation convention allows this to be rewritten as (g).   If $b = \infty$ so that the righthand component in (e) has coefficient $0$, a handleslide of the lefthand component splits the link and a subsequent slam dunk yields the two component unlink with surgery coefficients $s$ and $[k,1,-k+r]$.  In either case, $Y^*_k(-1,r,s,b) = L[s,-b-1,k,1-k,0,r]$.
}
\label{fig:secondsurgery}
\end{figure}

\subsection{LJP knots and the Whitehead link}

\begin{lemma}\label{lem:whitehead}
	For $r=0$ and $s\in\Z$, the manifold $Y(m,0,s,b)$ is $(m, s+1/b)$--surgery on the Whitehead link.  Furthermore, if $m \in \Z$, then for each $k\in\Z$ it contains the fibered JP framed knot $K_k(m,0,s,b)$ on which surgery produces the manifold $Y^*_k(m,0,s,b) = L[-k,m,k-1,-b-1,s]$.
\end{lemma}

\begin{figure}[h!]
	\includegraphics[width=\textwidth]{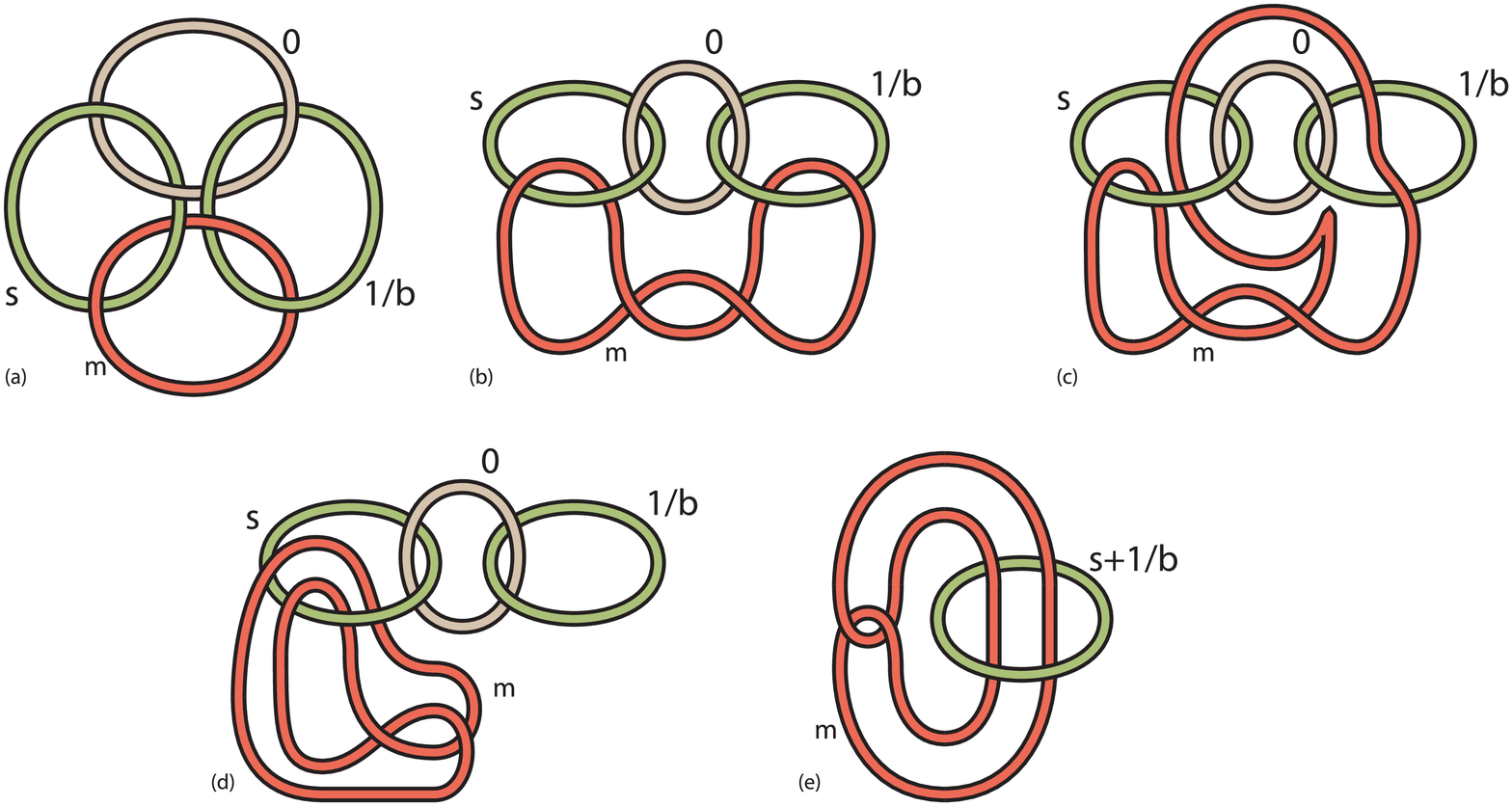}
	\caption{For $s \in\Z$, the manifold $Y(m,0,s,b)$ is $(m, s+1/b)$--surgery on the Whitehead link: Beginning with the surgery diagram of $Y(m,r,s,b)$ of Figure~\ref{fig:Ymrsb} and setting $r=0$ in (a), an isotopy produces (b). A handleslide over the $0$--framed component gives (c) and a further isotopy gives (d).  Assuming $s\in\Z$, two slam dunks then produce (e).  }
	\label{fig:Ym0sb}
\end{figure}

\begin{figure}[h!]
\includegraphics[width=\textwidth]{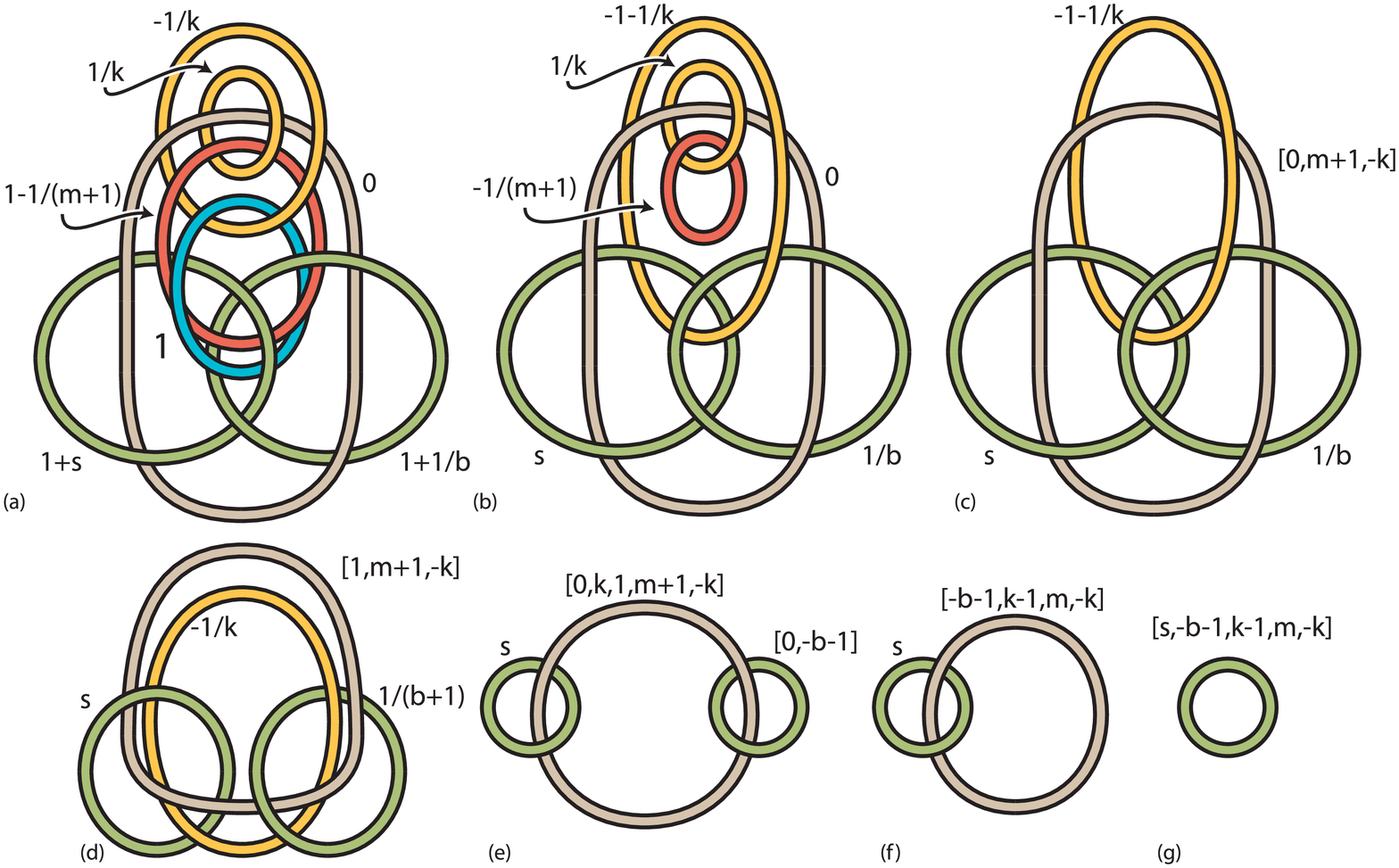}
\caption{With $r=0$ and  $m,s \in\Z$, framed surgery on $K_k(m,0,s,b)$ produces the manifold $Y^*_k(m,0,s,b) = L[-k,m,k-1,-b-1,s]$: Starting from the surgery diagram of $K_k(m,r,s,b)$ of  Figure~\ref{fig:surgerylink} in (a), we set $s=0$ and place the $+1$ framing on $K$.  Blowing down  $K$ produces (b).  A Rolfsen twist of $J$ into $L_+$ followed by a slam-dunk into $R$ gives (c).  A twist along $C_1$ with an isotopy gives (d) and absorbing the $-1/k$ twists of $C_1$ into the now parallel component $R$ gives (e). A Rolfsen twist (along with blowing down the $+1$ in the continued fraction) and then a slam-dunk yields (f) and (g).  }
\label{fig:cabledgofk}
\end{figure}

\begin{proof}
Through isotopies and Kirby calculus, Figure~\ref{fig:Ym0sb} shows how setting $r=0$ and $s\in\Z$ in the surgery description of $Y(m,r,s,b)$ given in Figure~\ref{fig:Ymrsb} yields a description of  $Y(m,0,s,b)$ as $(m, s+1/b)$--surgery on the Whitehead link.  If $m \in\Z$, then by Lemma~\ref{lem:jpcondition}  the knot $K_k(m,0)$ is JP in the twice-punctured torus bundle $M(m,0)$.  Hence $K_k(m,0,s,b)$ is fibered JP in $Y(m,0,s,b)$.

Next, Figure~\ref{fig:cabledgofk} shows that when $r=0$ and $m,s \in \Z$, the result of the framed surgery on $K_k(m,0,s,b)$ is the manifold $Y^*_k(m,0,s,b) = L[-k,m,k-1,-b-1,s]$.
\end{proof}

\begin{theorem}\label{thm:bulk2}
	For $k,b\in\Z$ and $(m,s) \in \{(-1,-6), (-2,-4), (-3,-3)\}$, the framed knot $K_k(m,0,s,b)$ is a fibered LJP knot in the lens space $Y(m,0,s,b)$ and framed surgery produces the lens space $Y^*_k(m,0,s,b)$ according to Table~\ref{table:cabledgofk}.
\begin{table}[h!]
	\caption{The lens spaces of Theorem~\ref{thm:bulk2}. }
	\label{table:cabledgofk}
	\small
		\begin{tabular}[t]{@{}llrr@{}}
			\toprule
   $Y(m,0,s,b)$ &   $Y^*_k(m,0,s,b)$  &$m$ & $s$  \\
	\midrule
 $L(6b-1,2b-1)$ & 	$L[-k,-1,k-1,-b-1,-6]$&	$-1$&$-6$ 	 \\
 $L(8b-2,2b-1)$ & 	$L[-k,-2,k-1,-b-1,-4]$	&$-2$&$-4$	 \\
 $L(9b-3,3b-2)$ & 	$L[-k,-3,k-1,-b-1,-3]$&$-3$&$-3$	 \\
			\bottomrule
		\end{tabular}
\end{table}

The knots $K_k(-1,0,-6,b)$, $K_k(-2,0,-4,b)$, and $K_k(-3,0,-3,b)$ are generically hyperbolic and asymmetric. 
\end{theorem}

\begin{proof}
From Lemma~\ref{lem:whitehead}, $Y(m,0,s,b)$ is $(m, s+1/b)$--surgery on the Whitehead link.  
This is a lens space if and only if  $(m,s+1/b)$ or $(s+1/b,m)$ is $(-1, -6+1/n)$, $(-2, -4+1/n)$, $(-3, -3+1/n)$, or $(p/q, \infty)$ for some $n \in \Z$ by \cite[Table A.5]{MP} (see also \cite{baker-optils}).
 
Assuming $m \in \Z$, Lemma~\ref{lem:whitehead} also shows that $K_k(m,0,s,b)$ is (fibered) JP in $Y(m,0,s,b)$.  (This assumption further implies that $m$--filling on the first component of the Whitehead link makes the exterior of the second component into a once-punctured torus bundle, cf.\ \cite[Propostion 3]{HMW}.)    
Finally, for the knot to be LJP so that $Y^*_k(m,0,s,b)$ is a lens space,  Lemma~\ref{lem:ljpandmjp} shows that we also need $b \in \Z$.  Thus if $Y(m,0,s,b)$ is to be a lens space for any $b \in \Z$, we must take $(m,s) \in \{(-1,-6),(-2,-4),(-3,-3)\}$. 

The claims of hyperbolicity and asymmetry follow from Lemma \ref{lem:bulk2symmetry}.
\end{proof}

\begin{addendum}\label{add:t12685}
	For $k\in\Z$ and $(m,s,b)$ as detailed in Table~\ref{table:cabledgofk2}, the framed knot $K_k(m,0,s,b)$ has a fibered longitudinal jointly primitive presentation in the lens space $Y(m,0,s,b)$ and framed surgery produces the lens space $Y^*_k(m,0,s,b)$ according to Table~\ref{table:cabledgofk2}.
	
\end{addendum}

\begin{table}[h!]
\small
	\caption{
	The lens spaces of Addendum~\ref{add:t12685}.  Where marked with \S, the knots $K_k(m,0,s,b)$ are strongly invertible by Lemma~\ref{lem:strinv} (and also by direct computation). 
	The remaining knots $K_k(m,0,s,b)$ are unmarked and are asymmetric in the sense that for most large values of the free parameters we can fill by fully-generic, symmetry-breaking multislopes as defined in \S\ref{sec:isometrygroupsfillings}. 
	}
	
	\label{table:cabledgofk2}
		\begin{tabular}[t]{@{}llrrrr@{}}
			\toprule
			$Y(m,0,s,b)$ &   $Y^*_k(m,0,s,b)$ & $m$ &  $s$ & $b$  \\
			\midrule
			$L(11,3)$ & 		$L[-k+1,k,1,-5]$	& 	$-1$&$-5$&$-2$ 	\\
			$L(13,5)$ & 		$L[ -k+1,k,-3, -7]$	& 	$-1$&$-7$&$2$ 	\\
			\midrule
			$L(14,3)$ & 		$L[-k,-2,k-1,1,-3]$	& 	$-2$&$-3$&$-2$  \\		
			$L(18,5)$ & 		$L[-k,-2,k-1,-3,-5]$& 	$-2$&$-5$&$2$ 	\\
			\midrule
			$L(15,4)$ & 		$L[-k,-3,k-1,1,-4]$	& 	$-3$&$-2$&$-2$ & \S \\	
			$L(21,8)$ & 		$L[-k,-3,k-1,-3,-2]$& 	$-3$&$-4$&$2$ 	\\
			\bottomrule\\
		\end{tabular}

		\begin{tabular}[t]{@{}llrrrr@{}}
			\toprule
			$Y(m,0,s,b)$ &   $Y^*_k(m,0,s,b)$ & $m$ & $s$ & $b$  \\
			\midrule
			$L(6,1)$   &      $L[-k,-2,k-1,0,-2]$&		$-2$&$-2$&$-1$ & \S\\
			$L(6,1)$   &      $L[-k,-2,k-1,-2,-4]$& 	$-2$&$-4$&$1$  \\
			$L(12,5)$   &      $L[-k,-4,k-1,0,-2]$& 	$-4$&$-2$&$-1$ & \S \\
			$L(12,5)$   &      $L[-k,-4,k-1,-2,-4]$& 	$-4$&$-4$&$1$ \\
			\midrule
			$L(6,1)$   &      $L[-k,-3,k-1,0,-1]$&		$-3$&$-1$&$-1$ & \S \\
			$L(6,1)$   &      $L[-k,-3,k-1,-2,-3]$& 	$-3$&$-3$&$1$ \\	
			$L(10,3)$   &      $L[-k,-5,k-1,0,-1]$&		$-5$&$-1$&$-1$ & \S\\
			$L(10,3)$   &      $L[-k,-5,k-1,-2,-3]$& 	$-5$&$-3$&$1$ \\
			\midrule		
			$L(5,1)$   &      $L[-k,-5,k-1,0,0]$&		$-5$&$0$&$-1$  & \S\\
			$L(5,1)$   &      $L[-k,-5,k-1,-2,-2]$& 	$-5$&$-2$&$1$ & \S\\
			$L(7,3)$   &      $L[-k,-7,k-1,0,0]$&		$-7$&$0$&$-1$ & \S\\
			$L(7,3)$   &      $L[-k,-7,k-1,-2,-2]$& 	$-7$&$-2$&$1$ & \S\\
			\bottomrule
		\end{tabular}
	
\end{table}

\begin{proof}
Continue from the proof of Theorem~\ref{thm:bulk2}.
Since $s + \frac{1}{\pm2} = (s \pm 1) + \frac{1}{\mp2}$, it follows from the surgery description in Figure~\ref{fig:Ym0sb} that $Y(m,0,s,\pm2) = Y(m,0,s\pm1,\mp2)$. Thus for each $(m,s\pm1) \in \{(-1,-6),(-2,-4),(-3,-3)\}$ we obtain a fibered LJP knot in a lens space for every $k \in \Z$ by setting $b = \mp2$.  This gives $6$ more families shown in Table~\ref{table:cabledgofk2}(Top).
	
Furthermore, taking $|b|=1$ and $|n|=1$ we may also choose integers $m$ and $s$ so that $(s+1/b,m) \in \{ (-1,-6\pm1), (-2,-4\pm1), (-3,-3\pm1)\}$.  Hence this gives $12$ more families of fibered LJP knots in lens spaces. These are shown in Table~\ref{table:cabledgofk2}(Bottom).

Lemma~\ref{lem:strinv} shows that half of the families of Table~\ref{table:cabledgofk2} necessarily produce strongly invertible knots.  These are indicated with the symbol $\S$.
\end{proof}

\begin{remark}\label{rem:bulk2extension}
\begin{enumerate}

\item 
Take $b = \infty$ so that $s+1/b = s$, and  take $(m,s) \in \{(-1,-6+1/c), (-2,-4+1/c), (-3,-3+1/c)\}$ for $c \in \Z$.  Then the framed knot $K_k(m,0,s,\infty)$ is a fibered MJP knot in the lens space $Y(m,0,s,\infty)$ and framed surgery produces the connected sum of lens space $Y^*_k(m,0,s,\infty)$.

\item Taking $s=b=\infty$, then for any $k,m \in \Z$, both $Y(m,0,\infty,\infty)$ and $Y^*_k(m,0,\infty,\infty)$ are lens spaces.  In this situation 
$C'_0$ becomes a genus one fibered knot where the twice-punctured torus $\Sigma$ caps off to a fiber.  Thus the knots $K_k(m,0,\infty,\infty)$ are doubly primitive knots.  
\end{enumerate}
\end{remark}


\appendix

\section{Early examples}
\label{sec:MMM_twice_drilled}

\begin{figure}
	\includegraphics[width=\textwidth]{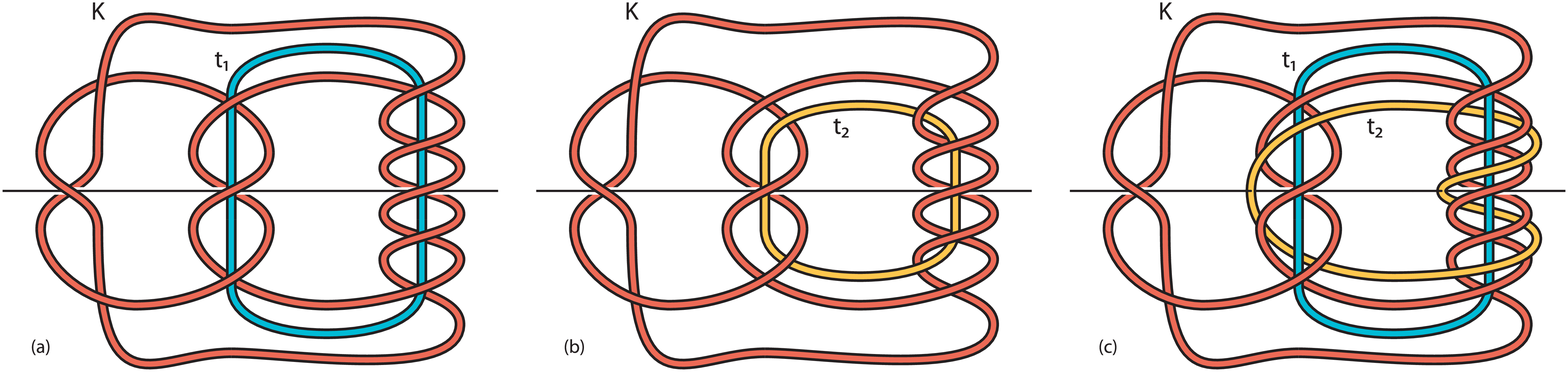}
	\caption{In (a) and (b) are copies of the $(-3,3,5)$ pretzel knot $K$ linked with unknots $t_1$ and $t_2$ respectively. See \cite[Figures 1 \& 9]{mattman2006seifert}. Performing $-1/n$--surgeries on the unknots $t_1$ and $t_2$ produces the two families of non-strongly invertible genus one knots $K_n$ and $K'_n$ of Mattman, Miyazaki, and Motegi for which $+1$--surgery yields a small Seifert fiber space surgery.    Observe that with respect to the genus $1$ Heegaard tori defined by $t_1$ and $t_2$, these knots are all $(1,2)$--knots. In (c) we combine (a) and (b) so that $t_1 \cup t_2$ is a $(2,6)$--torus link. One then finds that $+1$--surgery on $K$ in the cable space exterior of $t_1 \cup t_2$ is again a cable space.
	}
	\label{fig:MMM-knots}
\end{figure}

We observe that ``natural'' examples of one-cusped hyperbolic manifolds with tunnel number $2$ and two lens space fillings have been within reach for several years.

Preceding \cite{DMM}, Mattman-Miyazaki-Motegi describe the first examples of non-strongly invertible hyperbolic knots in $S^3$ for which some surgery produces a small Seifert fibered space \cite{mattman2006seifert}.  Their two families of knots $K_n$ and $K_n'$ giving these examples are obtained through twisting the $(-3,3,5)$ pretzel knot $K$ with its $+1$--surgery by $-1/n$--surgery on one of two unknots $t_1$ and $t_2$, detailed in Figure~\ref{fig:MMM-knots}(a) and (b).   Later, Gainullin observes that $15/4$--surgery on $t_1$ (the slope of a regular fiber about $t_1$ after the Seifert fibered $+1$--surgery on $K$) instead produced a hyperbolic knot in a lens space with a reducible surgery that did not fit previously known constructions \cite{gainullin}.  

Noting that the slopes $(15j+4)/(4j+1)$ have distance $1$ from $15/4$ for any $j \in \Z$, $(15j+4)/(4j+1)$--surgery on $t_1$ sends $K$ to a knot in a lens space with a lens space surgeries.   Appealing to the computations in \cite{mattman2006seifert}, the symmetry group of the hyperbolic complement of $K \cup t_1$ is not strongly invertible on the cusp of $K$. Therefore it follows from Thurston's Hyperbolic Dehn Surgery Theorem and Theorem~\ref{thm:fillinglengths}
 that for all but finitely many $j \in \Z$, the result of the $(15j+4)/(4j+1)$--filling on the $t_1$ cusp of the complement of $K \cup t_1$ is a one-cusped hyperbolic manifold that is not strongly invertible.  Therefore it has tunnel number at least $2$.  Moreover it has two lens space fillings, and so neither core of these lens space fillings can be a $(1,1)$--knot. 
In fact, as one observes from Figure~\ref{fig:MMM-knots}(a) that the image of $K$ is at worst a $(1,2)$--knot in any lens space obtained by surgery on $t_1$, it must be a $(1,2)$--knot (with tunnel number exactly $2$) in this situation.  A similar discussion applies for the images of $K$ upon $(12j+5)/(5j+2)$--surgery on $t_2$.

\bigskip
	
Figure \ref{fig:MMM-knots}(c) unifies (a) and (b) into a three-component hyperbolic link $K \cup t_1 \cup t_2$. (Snappy identifies the complement of this link as the census manifold $o9_{44012}$, which is also the complement of the exterior of the link $L11n366$.)   The $1/0$ and $+1$--fillings on the cusp of $K$ each produce cable spaces homeomorphic to $A(3/1)$, but with opposite orientations.  As $M(-1,-2) = A(3/1)$ and Lemma~\ref{lem:cables} shows that $M_{-2}^*(-1,-2) = A(3/2)$, one checks and confirms that the complement of $K\cup t_1 \cup t_2$ is also our manifold $M_{-2}(-1,-2)$.

One identifies the family of surgeries between lens space induced by $(15j+4)/(4j+1)$--surgery on $t_1$ in Figure~\ref{fig:MMM-knots}(a) as the family $K_{-2}(-1,-2,-5/2,j)$.  
Similarly, the family of surgeries between lens space induced by $(12j+5)/(5j+2)$--surgery on $t_2$ in Figure~\ref{fig:MMM-knots}(b) may be identified as the family $K_{-2}(-1,-2,-4,j)$.

\section{LJP knots and the Magic Manifold}\label{sec:magic}

This appendix contains the data for vastly many more exceptional fillings of $K_k(m, r, s, b)$ than described in the main body of the paper.  Indeed, as one may observe, the exterior of the sublink $C_0' \cup C_1' \cup R \subset \calL$ in Figure~\ref{fig:surgerylink} is the homeomorphic to the ``Magic Manifold'' $N$ of Martelli and Petronio \cite{MP}.   In particular, our manifolds $Y(-1,r,s,b)$ are fillings of $N$: that is,  $Y(-1,r, s, b)=N(\alpha,\beta,\gamma)$, where $(\alpha,\beta,\gamma)$ is some permutation of $(r,1+s,1+1/b)$.

Table~\ref{table:mmfillings} identifies the fillings of $N$ which are lens spaces or connected sums of lens spaces according to \cite{MP}.  (Note that these lens spaces and lens space summands are only given up to homeomorphism.) 

Table~\ref{table:pairsoffillings} specifies the various choices of $(r,s,b)$  for which of the manifolds of Table~\ref{table:mmfillings} can be realized as $Y(-1,r, s, b)$ subject to the restriction that $b \in \Z$ or $b=\infty$ so that $Y^*_k(-1,r,s,b)$ will be a lens space or a connected sum of lens spaces. (Only $N(-\frac{3}{2}, -\frac{5}{2},-2)$ and $N(-\frac{3}{2}, -\frac{5}{2},-1)$ cannot be realized in this manner.) The resulting manifolds $Y^*_k(-1,r,s,b)$ are also given (as determined by the Kirby calculus of Figure~\ref{fig:firstsurgery}).
 
Tables~\ref{table:lenslens}, \ref{table:lensred}, \ref{table:redlens}, and \ref{table:redred} then separate the information of Table~\ref{table:pairsoffillings} into according to whether $Y(-1,r,s,b)$ and $Y^*_k(-1,r,s,b)$ are lens spaces or connected sums of lens spaces as shown in the chart below.  In Table~\ref{table:lenslens}, the two sets of parameters that produce the knots of Theorem~\ref{thm:bulk} are noted.

\begin{center}
\begin{tabular}{@{}lll@{}}
\toprule
 $Y(-1,r,s,b)$ & $Y^*_k(-1,r,s,b)$ &  Table\\
\midrule
Lens & Lens & \ref{table:lenslens} \\
Lens & Lens \# Lens &\ref{table:redlens} \\
Lens \# Lens & Lens & \ref{table:lensred} \\
Lens \# Lens & Lens \# Lens & \ref{table:redred} \\
\bottomrule
\end{tabular}
\end{center}

\begin{table}
\centering
\caption{The lens space and reducible fillings of the magic manifold $N$ except when one filling coefficient is $\infty$; obtained from \cite[Theorem~1.3 and Tables 2,3,4]{MP}.  Here, $n,m,t,u\in\Z$ and $t,u$ are coprime. Note that the lens spaces and lens space summands were  determined there only up to homeomorphism, see \cite[Section 1.1, Lens spaces]{MP}.  We identified  $N(-1,-3+\tfrac1n,\tfrac{t}{u})$ as being listed with the opposite orientation, and have corrected it below. However there may still be others listed with the opposite orientation.
 }
\label{table:mmfillings}
\begin{tabular}{@{}ll@{}}
\toprule
filling of magic manifold &  resulting manifold\\
\midrule
$N(-3,-1,\tfrac{t}{u})$  &  $L(2,1) \# L(t+3u, u)$ \\
$N(-3,-2,\tfrac{t}{u}) $ & $L(5t+7u, 2t+3u)$\\
$N(-3,-1 + \tfrac{1}{n}, -1 + \tfrac{1}{m})$ & $ L((2n+1)(2m+1)-4, (2n+1)m-2)$\\
$N(-2, -2, \tfrac{t}{u})$ & $ L(3, 1)\#L(t+2u, u)$\\
$N(-2, -2 + \tfrac{1}{n}, \tfrac{t}{u}) $ & $ L(3n(t+2u)-2t-u, n(t+2u)-t-u)$\\
$N(-1, -3 + \tfrac{1}{n}, \tfrac{t}{u}) $ & $L(2n(t+3u)-t-u, n(t+3u)-t-2u)$\\
$N(0, n, -4 - n) $ & $ L(2,1) \# L(3,1)$\\
$N(0, n, -4 - n + \tfrac{1}{m}) $ & $ L(6m-1, 2m-1)$\\
$N( -\tfrac{3}{2}, -\tfrac{5}{2}, -2) $ & $ L(2,1)$\\
$N( -\tfrac{3}{2}, -\tfrac{5}{2},-1)$ & $ L(13, 5)$\\
$N(-4, -\tfrac{1}{2}, -1) $ & $ L(11, 3)$\\
$N(-4, -\tfrac{1}{2}, 0) $ & $ L(13,5)$\\
\bottomrule
\end{tabular}
\end{table}

\begin{table*}
\centering
\caption{The lens space and reducible fillings of the magic manifold $N$ that may be obtained as $Y(-1,r,s,b)$.  The underlined terms correspond to those filled with $1+1/b$ for some $b \in \hatZ$.  The corresponding possibilities of parameters $(r,s,b)$ and the filled manifolds  $Y(-1,r,s,b)$ and $Y^*_k(-1,r,s,b)$ are given.  Here $n\in\Z$, and $b\in\hatZ$. (The parameter $m$ has been replaced with $n$ since now only one appears at a time.) Where marked with \S, the knots $K_k(-1,r,s,b)$ are strongly invertible by Lemma~\ref{lem:strinv}.}
\label{table:pairsoffillings}
\begin{tabular}{@{}llllr@{}}
\toprule
magic manifold filling & $(r,s,b)$ &  $Y(-1,r,s,b)$ & $Y^*_k(-1,r,s,b)$ \\
\otoprule
$N(-3,-1,\underline{1+\tfrac{1}{b}}) $& $(-3,-2,b)$ &  $ L[2] \# L[4,-b]$ & $L[-k-2,k,-b-1, -2]$   &\S\\ 
                                      &$(-1,-4,b)$&  $ L[2] \# L[4,-b]$ &  $L[-k,k,-b-1, -4]$  \\  \midrule
$N(-3,-2,\underline{1+\tfrac{1}{b}}) $& $(-3,-3,b)$  & $L[2, -2, 2, -b]$ & $L[-k-2,k,-b-1, -3]$\\
                                      &$(-2,-4, b)$ & $L[2, -2, 2, -b]$ & $L[-k-1,k,-b-1, -4]$ \\  \midrule
$N(-3,-1 + \tfrac{1}{n}, \underline{0})$ &$(-3, -2+\tfrac{1}{n},-1)$ & $ L[3, 1-n, 2]$ &$L[-k-2,k-2,-n]$ &\S\\
                                      &$(-1+\tfrac1n,-4, -1)$ & $ L[3, 1-n, 2]$ &  $L[-n,-k,k-4]$&\S\\  \midrule
$N(-2, -2, \underline{1+\tfrac{1}{b}})$ &$(-2,-3,b)$ & $L[3]\#L[3,-b]$& $L[-k-1,k,-b-1, -3]$\\
                                      & $(-2,-3, b)$& $L[3]\#L[3,-b]$& $L[-k-1,k,-b-1, -3]$\\  \midrule
$N(-2, -2 + \tfrac{1}{n}, \underline{1+\tfrac{1}{b}}) $&$(-2,-3+\tfrac{1}{n},b)$ & $L[3,1-n,3,-b]$ &  $L[-k-1,k,-b-1, -3,-n]$ \\
                                      &$(-2+\tfrac1n,-3,b)$& $L[3,1-n,3,-b]$ &$L[-n, -k-1, k,-b-1, -3]$\\  \midrule

$N(-1, -3 + \tfrac{1}{n}, \underline{1+\tfrac{1}{b}}) $ &$(-1, -4+\tfrac{1}{n},b)$ & $L[2,n,-4,b]$ & $L[-k,k,-b-1, -4,-n]$\\
                                      &$(-3+\tfrac1n,-2,b)$ & $L[2,n,-4,b]$ &$L[-n,-k-2,k,-b-1, -2]$ &\S\\  \midrule
$N(\underline{0}, n, -4 - n) $  &$(n, -5-n,-1)$ & $L[2] \# L[3]$ &$L[-k+n+1,k-n-5]$ &\S\\
                                      &$(-4-n, n-1, -1)$ & $L[2] \# L[3]$ &  $L[-k-n-3,k -1]$ &\S \\  \midrule
                                      
$N(0, \underline{0}, -4) $ &$(0,-5,-1)$ & $L[2] \# L[3]$ & $L[-k+1,k -1]$ &\S \\
                                     &$(-4,-1,-1)$ & $L[2] \# L[3]$ & $L[-k-3,k-1]$&\S \\  \midrule
                                     
$N(0, \underline{1}, -3) $&$(0,-4,\infty)$ & $L[2] \# L[3]$ &  $L[-k+1,k]\#L[-4]$&\S\\
                                    & $(-3,-1,\infty)$ & $L[2] \# L[3]$ &$L[-k-2,k]$&\S \\  \midrule
                                    
$N(0, \underline{2}, -2) $&$(0,-3,1)$ & $L[2] \# L[3]$ &$L[-k+1,k,-2, -3]$\\
                                    &$(-2,-1,1)$ & $L[2] \# L[3]$ &$L[-k-1,k+1]$ &\S\\  \midrule
                                    
$N(0, \underline{0}, -4 + \tfrac{1}{n}) $ &$(0,-5+\tfrac{1}{n},-1)$ & $L[3,1-n, 2]$ &$L[-k+1,k-5,-n]$ &\S\\
                                   &$(-4+\tfrac1n,-1,-1)$ & $L[3,1-n, 2]$ &$L[-n,-k-3,k-1]$&\S \\  \midrule

$N(0,\underline{1}, -3 + \tfrac{1}{n}) $ &$(0,-4+\tfrac{1}{n},\infty)$ & $L[3,1-n, 2]$ & $L[-k+1,k]\#L[-4,-n]$&\S\\
                                     &$(-3+\tfrac1n,-1,\infty)$ & $L[3,1-n, 2]$ &$L[-n,-k-2,k]$ &\S\\  \midrule
                                     
$N(0, \underline{2}, -2 + \tfrac{1}{n}) $ &$(0,-3+\tfrac{1}{n},1)$ & $L[3,1-n, 2]$ &$L[-k+1,k,-2,-3,-n]$ \\
                                    &$(-2+\tfrac1n,-1,1)$  & $L[3,1-n, 2]$ &$L[-n,-k-1,k+1]$ &\S\\  \midrule
                                    
$N(0, -4, \underline{\tfrac{1}{2}}) $  &$(0,-5,-2)$ & $L[3,-1,2]$ &$L[-k+1,k,1,-5 ]$ \\
                            		&$(-4,-1,-2)$ & $L[3,-1,2]$ & $L[-k-3,k,2 ]$&\S\\  \midrule

$N(0, -5,  \underline{1+\tfrac{1}{b}}) $  &$(0,-6,b)$ & $L[3,1-b, 2]$ &$L[-k+1,k,-b-1, -6]$ \\
                                          &$(-5,-1,b)$ & $L[3,1-b, 2]$ & $L[-k-4,k,-b]$&\S\\  \midrule

$N(0, -6, \underline{\tfrac{3}{2}}) $  &$(0,-7,2)$ & $L[3,3, 2]$ &$L[-k+1,k,-3,-7]$ \\
     				                   &$(-6,-1,2)$ & $L[3,3, 2]$ & $L[-k-5,k,-2 ]$&\S\\  \midrule
$N(-4, -\tfrac12, \underline{0}) $  &$(-4, -3/2,-1)$ & $L[3,3,2]$ &$L[-k-3,k-2,-2]$&\S\\
                                    &$(-\tfrac12,-5,-1)$ & $L[3,3,2]$ & $L[2 ,-k+1,k-1]$&\S\\
\bottomrule
\end{tabular}
\end{table*}

\begin{remark}\label{rem:MMfillingtolensfilling}
Every lens space $Y(-1,r,s,b)$ appearing in Table~\ref{table:pairsoffillings} may be described with one of the two following forms  for some $x,y \in \Z$: 
\begin{align*}
L[3,x,3,y] &= L(3x(1-3y)+6y-1,x(1-3y)+y) \\
L[2,x,4,y]&= L(2x(1-4y)+6y-1, x(1-4y)+y)
\end{align*}
Consequently, we see that the some of the knots $K_k(m,0,s,b)$ of Table~\ref{table:cabledgofk} with $m\neq -1$ are not obtained in this appendix.   For example, observe from Table~\ref{table:cabledgofk} that $Y(-2,0,-4,-1) = L(6,1)$ and $Y^*_k(-2,0,-4,-1) = L(14 k^2 - 6 k + 3, -14 k -1)$.  Since the order of  $L[3,x,3,y]$ mod $3$ and the order of $L[2,x,4,y]$ mod $2$ are both $-1$, neither can be $Y(-2,0,-4,-1)$. Similarly choosing $k \equiv 0 \mod 3$, $Y^*_k(-2,0,-4,-1)$ cannot be $L[3,x,3,y]$.  Then, as one may check,  there are no integers $x,y$ such that either lens space $Y^*_3(-2,0,-4,-1) = L(111,68)$ or $Y^*_6(-2,0,-4,-1)=L(471,386)$ is homeomorphic to the lens space $L[2,x,4,y]$. 
\end{remark}

\begin{table*}
\centering
\caption{ With the parameters $(r,s,b)$ as given, the filled manifolds $Y(-1,r,s,b)$ and $Y^*_k(-1,r,s,b)$ are the lens spaces indicated for any $k,n,b \in \Z$.
 In the 5th and 7th rows, the asterisk on the parameters indicates that they correspond to the surgeries of Theorem~\ref{thm:bulk}. For generic $k,n,b \in \Z$, the knots $K_k(-1,r,s,b)$ are hyperbolic. Where marked with \S, the knots $K_k(-1,r,s,b)$ are strongly invertible by Lemma~\ref{lem:strinv}.}
\label{table:lenslens}
\begin{tabular}{@{}llllr@{}}
\toprule
$(r,s,b)$ & $Y(-1,r,s,b)$ & $Y^*_k(-1,r,s,b)$  \\
\midrule
 $(-3,-3,b)$  & $L[2, -2, 2, -b]$ & $L[-k-2,k,-b-1, -3]$ &   \\
 $(-2,-4, b)$ & $L[2, -2, 2, -b]$ & $L[-k-1,k,-b-1, -4]$ &   \\
 
$(-3, -2+\tfrac{1}{n},-1)$ & $ L[3, 1-n, 2]$ &$L[-k-2,k-2,-n]$ & \S \\
$(-1+\tfrac1n,-4, -1)$ & $ L[3, 1-n, 2]$ &  $L[-n,-k,k-4]$ & \S\\

$(-2,-3+\tfrac{1}{n},b)^*$ & $L[3,1-n,3,-b]$ &  $L[-k-1,k,-b-1, -3,-n]$ &  \\
$(-2+\tfrac1n,-3,b)$& $L[3,1-n,3,-b]$ & $L[-n, -k-1, k,-b-1, -3]$&  \\

$(-1, -4+\tfrac{1}{n},b)^*$ & $L[2,n,-4,b]$ & $L[-k,k,-b-1, -4,-n]$&  \\
$(-3+\tfrac1n,-2,b)$ & $L[2,n,-4,b]$ & $L[-n,-k-2,k,-b-1, -2]$& \S\\

$(0,-5+\tfrac{1}{n},-1)$ & $L[3,1-n, 2]$ & $L[-k+1,k-5,-n]$ & \S\\
$(-4+\tfrac1n,-1,-1)$ & $L[3,1-n, 2]$ & $L[-n,-k-3,k-1]$ & \S\\

$(-3+\tfrac1n,-1,\infty)$ & $L[3,1-n, 2]$ & $L[-n,-k-2,k]$& \S\\

$(0,-3+\tfrac{1}{n},1)$ & $L[3,1-n, 2]$ & $L[-k+1,k,-2,-3,-n]$& \\
$(-2+\tfrac1n,-1,1)$  & $L[3,1-n, 2]$ & $L[-n, ,-k-1,k+1]$ & \S \\

$(0,-5,-2)$ & $L[3,-1,2]$ &$L[-k+1,k,1,-5 ]$ &  \\
$(-4,-1,-2)$ & $L[3,-1,2]$ & $L[-k-3,k,2 ]$ &  \S \\  

$(0,-6,b)$ & $L[3,1-b, 2]$ & $L[-k+1,k,-b-1, -6]$ &  \\
$(-5,-1,b)$ & $L[3,1-b, 2]$ & $L[-k-4,k,-b]$& \S\\

$(0,-7,2)$ & $L[3,3, 2]$ &$L[-k+1,k,-3,-7]$ &  \\
$(-6,-1,2)$ & $L[3,3, 2]$ & $L[-k-5,k,-2 ]$ &  \S \\

$(-4, -3/2,-1)$ & $L[3,3,2]$ & $L[-k-3,k-2,-2]$& \S\\
$(-\tfrac12,-5,-1)$ & $L[3,3,2]$ & $L[2 ,-k+1,k-1]$& \S\\
\bottomrule
\end{tabular}
\end{table*}

\begin{table*}
\centering
\caption{
	 With the parameters $(r,s,b)$ as given, the filled manifolds $Y(-1,r,s,\infty)$ and $Y^*_k(-1,r,s,\infty)$ are the lens spaces and connected sums of lens spaces indicated respectively  for any $k,n \in \Z$.  
	 For generic $k,n \in \Z$, the knots $K_k(-1,r,s,\infty)$ are hyperbolic and strongly invertible. The strong invertibility (still marked with \S) is given by Lemma~\ref{lem:strinv}. }
\label{table:lensred}
\begin{tabular}{@{}llllr@{}}
\toprule
$(r,s,b)$ & $Y(-1,r,s,b)$ & $Y^*_k(-1,r,s,b)$  \\
\midrule
 $(-3,-3,\infty)$  & $L[2, -2, 2]$ & $L[-k-2,k]\#L[-3]$& \S\\
 $(-2,-4, \infty)$ & $L[2, -2, 2]$ & $L[-k-1,k]\#L[-4]$& \S\\
$(-2,-3+\tfrac{1}{n},\infty)$ & $L[3,1-n,3]$ &  $L[-k-1,k]\#L[ -3,-n]$ & \S\\
$(-2+\tfrac1n,-3,\infty)$& $L[3,1-n,3]$ &$L[-n, -k-1, k]\#L[ -3]$& \S\\
$(-1, -4+\tfrac{1}{n},\infty)$ & $L[2,n,-4]$ & $L[-k,k]\#L[ -4,-n]$& \S\\
$(-3+\tfrac1n,-2,\infty)$ & $L[2,n,-4]$ &$L[-n,-k-2,k]\#L[ -2]$ & \S\\
$(0,-4+\tfrac{1}{n},\infty)$ & $L[3,1-n, 2]$ & $L[-k+1,k]\#L[-4,-n]$& \S\\
\bottomrule
\end{tabular}
\end{table*}

\begin{table*}
\centering
\caption{ 
	 With the parameters $(r,s,b)$ as given, the filled manifolds $Y(-1,r,s,\infty)$ and $Y^*_k(-1,r,s,\infty)$ are the connected sums of lens spaces and  lens spaces indicated respectively  for any $k,n,b \in \Z$. 
 For generic $k,n,b \in \Z$, the knots $K_k(-1,r,s,b)$ are hyperbolic.  Where marked with \S, the knots $K_k(-1,r,s,b)$ are strongly invertible by Lemma~\ref{lem:strinv}.}
\label{table:redlens}
\begin{tabular}{@{}llllr@{}}
\toprule
$(r,s,b)$ & $Y(-1,r,s,b)$ & $Y^*_k(-1,r,s,b)$ \\
\midrule
$(-3,-2,b)$ &  $ L[2] \# L[4,-b]$ & $L[-k-2,k,-b-1, -2]$ & \S\\
$(-1,-4,b)$&  $ L[2] \# L[4,-b]$ &  $L[-k,k,-b-1, -4]$  & \\
 $(-2,-3,b)$ & $L[3]\#L[3,-b]$& $L[-k-1,k,-b-1, -3]$  &  \\
$(n, -5-n,-1)$ & $L[2] \# L[3]$ &$L[-k+n+1,k-n-5]$ & \S \\
$(-4-n,n-1,-1)$ & $L[2] \# L[3]$ &  $L[-k-n-3,k -1]$ & \S \\
$(0,-5,-1)$ & $L[2] \# L[3]$ & $L[-k+1,k -1]$ & \S\\
 $(-4,-1,-1)$ & $L[2] \# L[3]$ & $L[-k-3,k-1]$  & \S \\
 $(-3,-1,\infty)$ & $L[2] \# L[3]$ &$L[-k-2,k]$  & \S \\
$(0,-3,1)$ & $L[2] \# L[3]$ &$L[-k+1,k,-2, -3]$  &  \\
$(-2,-1,1)$ & $L[2] \# L[3]$ &$L[-k-1,k+1]$  & \S \\
$(-5,-1,\infty)$ & $L[3]\#L[ 2]$ & $L[-k-4,k]$  & \S \\
\bottomrule
\end{tabular}
\end{table*}

\begin{table*}
\centering
\caption{
	 With the parameters $(r,s,b)$ as given, the filled manifolds $Y(-1,r,s,\infty)$ and $Y^*_k(-1,r,s,\infty)$ are  connected sums of lens spaces  for any $k \in \Z$.
	 For generic $k \in \Z$, the knots $K_k(-1,r,s,\infty)$ are hyperbolic and strongly invertible. The strong invertibility (still marked with \S) is given by Lemma~\ref{lem:strinv}. }
\label{table:redred}
\begin{tabular}{@{}llllr@{}}
\toprule
$(r,s,b)$ & $Y(-1,r,s,b)$ & $Y^*_k(-1,r,s,b)$  \\
\midrule
$(-3,-2,\infty)$ &  $ L[2] \# L[4]$ & $L[-k-2,k] \# L[ -2]$   & \S\\
$(-1,-4,\infty)$&  $ L[2] \# L[4]$ &  $L[-k,k] \# L[-4]$   & \S\\
$(0,-4,\infty)$ & $L[2] \# L[3]$ &  $L[-k+1,k]\#L[-4]$  & \S\\
 $(-2,-3,\infty)$ & $L[3]\#L[3]$& $L[-k-1,k] \# L[-3]$  & \S\\
$(0,-6,\infty)$ & $L[3]\#L[ 2]$ &$L[-k+1,k]\#L[ -6]$  & \S\\
\bottomrule
\end{tabular}
\end{table*}

\begin{footnotesize}
	\bibliographystyle{alpha}
	\bibliography{BHL-JP-bib}
\end{footnotesize}

\end{document}